%% file: nonlinear_hhj_tdnns.tex
\documentclass[12pt, oneside, reqno]{amsart}

\input{pre.tex}

\title[HHJ \& TDNNS for shells]{The Hellan--Herrmann--Johnson and TDNNS method for linear and nonlinear shells}

\author[M.~Neunteufel]{Michael~Neunteufel}
\address{Institute for Analysis and Scientific Computing, TU Wien, Wiedner Hauptstr. 8-10, 1040 Wien, Austria}
\email{michael.neunteufel@asc.tuwien.ac.at}

\author[J.~Sch\"oberl]{Joachim~Sch\"oberl}
\address{Institute for Analysis and Scientific Computing, TU Wien, Wiedner Hauptstr. 8-10, 1040 Wien, Austria}
\email{joachim.schoeberl@tuwien.ac.at}

\begin{document}

\begin{abstract}
	In this paper we extend the recently introduced mixed Hellan--Herrmann--Johnson (HHJ) method for nonlinear Koiter shells to nonlinear Naghdi shells by means of a hierarchical approach. The additional shearing degrees of freedom are discretized by H(curl)-conforming N\'ed\'elec finite elements entailing a shear locking free method. By linearizing the models we obtain in the small strain regime linear Kirchhoff--Love and Reissner--Mindlin shell formulations, which reduce for plates to the originally proposed HHJ and TDNNS method for Kirchhoff--Love and Reissner--Mindlin plates, respectively. By using the Regge interpolation operator we obtain locking-free arbitrary order shell methods. Additionally, the methods can be directly applied to structures with kinks and branched shells. Several numerical examples and experiments are performed validating the excellence performance of the proposed shell elements.
	\\
	\vspace*{0.25cm}
	
\end{abstract}

\maketitle

\section{Introduction}
\label{sec:intro}

This paper is concerned with the extension of the recently proposed Hellan--Herrmann--Johnson (HHJ) method for nonlinear Koiter to Naghdi shells and the linearization in the context of Kirchhoff--Love and Reissner--Mindlin plates and shells.

For thin-walled structures the Kirchhoff--Love hypothesis is frequently assumed to eliminate the shearing degrees of freedom leading to a fourth order problem. Due to the difficulty of the construction of $C^1$-conforming Kirchhoff--Love plate and shell elements, several approaches have been considered. An (incomplete) list includes discrete Kirchhoff \cite{BZH01}, non-conforming \cite{Mor71,KB96}, and rotation-free elements \cite{OZ00}, as well as discontinuous Galerkin \cite{EGHL02,HLM2017}, Isogeometric Analysis \cite{HCB05,SF19a}, $C^1$-continuous Trace-Finite-Cell \cite{Gfr2021}, and mixed methods \cite{RZ19,NS19}. For the latter the moment stress tensor is used as additional unknown reducing the fourth order to a second order mixed saddle-point problem. The HHJ method \cite{Hel67,Her67,Joh73, Com89}, allowing for linear and high-order formulations, has been generalized from linear Kirchhoff--Love plates to nonlinear Koiter shells in \cite{NS19}. In this work we present an improved linearization analysis deriving the HHJ method for linear Kirchhoff--Love shells. The involved fields enable the handling of structures with kinks and branched shells without additional treatment as e.g. enforcing $C^1$-continuity by penalties \cite{VPS2017} or conditions for angle preservation over patches \cite{KBLW09}. For a continuous but non-smooth surface its normal vector is not globally continuous. The curvature computation involving the first derivatives of the normal vector has then to be performed in distributional sense. Motivated by discrete differential geometry it consists of an element-wise Weingarten tensor and the dihedral angle at element edges. By introducing the difference of curvatures of the initial and deformed configuration as additional unknown, a lifting of the distributional curvatures to a regular field is performed. This leads to an equivalent three-field Hu--Washizu formulation of the HHJ method. Further, it enables the usage of nonlinear material laws, which would not be applicable within a Hellinger--Reissner two-field formulation, where the inversion of the material law is mandatory.

In the moderate thickness regime the Reissner--Mindlin/Naghdi shells include additional shearing or rotational degrees of freedom (dofs) to better reflect the 3D behavior of the material. The additional dofs circumvent the fourth order problem enabling the usage of e.g. continuous Lagrangian finite elements. If not treated carefully, however, so-called shear locking might occur. In the literature a huge amount of procedures and strategies have been proposed to successfully tackle this problem. We use a hierarchical approach as in \cite{EOB13,OSRB2017}, where this problem is intrinsically avoided. We propose a novel method for nonlinear Naghdi shells by introducing tangential continuous N\'ed\'elec elements as shearing unknowns on the top of the HHJ method. The resulting formulation will be linearized to linear Reissner--Mindlin shells and plate. This reveals that the approach is an extension of the tangential-displacement normal-normal-stress continuous (TDNNS) method for Reissner--Mindlin plates \cite{PS11,PS17}.

To circumvent the problem of membrane locking we consider the procedure proposed in \cite{NS21} by inserting the interpolation operator into so-called Regge finite elements in the membrane energy term. The used finite elements are available for triangles and quadrilaterals such that unstructured triangular meshes can be combined with anisotropic quadrilateral elements in a locking-free manner. We show by means of several well-established (non-)linear benchmark examples the excellent performance of the proposed methods.

This paper is structured as follows. In the next section notations used throughout this work are introduced. A coordinate free description of surface differential operators by means of tangential differential calculus (TDC) is given and the shell equations and function spaces are stated. Section~\ref{sec:HHJ_nonlinear_Koiter_shells} is devoted to the derivation of the HHJ method for nonlinear Koiter shells. A three-field Hu--Washizu formulation is used to incorporate the distributional curvature approximation and then reduced to the two-field formulation as originally proposed in \cite{NS19}. The extension to nonlinear Naghdi shells is presented in Section~\ref{sec:TDNNS_nonlinear_Naghdi_shells}. In Section~\ref{sec:linearization} we linearize the nonlinear formulations in the context of shells and plates. Computational aspects including a stable angle computation and handling of nonlinear material laws are discussed in Section~\ref{sec:computational aspects}. Section~\ref{sec:kinks_branched_shells} is devoted to structures with kinks and branched shells. In Section~\ref{sec:finite_elements} the finite element discretization of the proposed methods are presented and in Section~\ref{sec:numerics} several numerical examples validate the performance of the methods.\\

\section{Notation and preliminaries}
\label{sec:notation}

This section provides definitions that we use throughout. For two vectors $u,v\in\R^3$ their scalar, outer, and cross products, respectively, are denoted by $u\cdot v$, $u\otimes v$, and $u\times v$. For two matrices $\bA,\bB\in\R^{3\times3}$ their Frobenius inner and tensor cross product is given by $\bA:\bB := \tr{\bA^\top \bB}:=\sum_{i,j=1}^3\bA_{ij}\bB_{ij}$ and $(\bA\tcross\bB)_{ij} := \varepsilon_{ikl}\varepsilon_{jmn}\bA_{km}\bB_{ln}$ \cite{BGO16}, where $\varepsilon_{ijk}$ is the Levi--Civita symbol and we made use of Einstein's summation convention of repeated indices.  $\bI\in \R^{3\times 3}$ will denote the identity matrix. The Euclidean and Frobenius norms are given by $\|v\|_2^2:=v\cdot v$, $\|\bA\|_F^2:=\bA:\bA$ and for a fourth order tensor identified as a linear mapping between matrices $\MT:\R^{3\times 3}\to\R^{3\times 3}$ we define the norm $\|\bA\|_{\MT}^2:=(\MT\bA):\bA$. If no misunderstandings are possible we neglect the subscripts of the Euclidean and Frobenius norm.  The angle between vectors $u,v$ is abbreviated by $\sphericalangle(u,v):=\arccos\big(\frac{u\cdot v}{\|u\|\|v\|}\big)$.

\subsection{Surfaces and tangential differential calculus}
\label{subsec:TDC}
Let $\rSurf\subset\R^3$ be an oriented surface, i.e., a two-dimensional manifold embedded in $\R^3$, with its globally continuous unit normal vector $\rnv:\rSurf\to \Stwo:=\{v\in\R^3\,\vert\, \|v\|=1 \}$. At each point $p\in\rSurf$ its tangent space is given by $T_p\rSurf=\{v\in\R^3\,\vert\, v\cdot\rnv(p)=0\}$ and we denote the tangent bundle by $T\rSurf=\bigcup_{p\in\rSurf}\{p\}\times T_p\rSurf$. The orthogonal complement of a vector $v\in \R^3$ is $\{v\}^\perp=\{u\in\R^3\,\vert\, u\cdot v=0\}$ such that $T_p\rSurf=\{\rnv(p)\}^\perp$, and identities like $\{\rnv\}^\perp=T\rSurf$ have to be understood to hold point-wise. For a scalar field $f:\rSurf\to \R$ we define the surface gradient $\rgrad f:\rSurf\to T\rSurf$, interpreted as column vector, in terms of tangential differential calculus (TDC) \cite{DZ11,Sp79,SF19a,SF19b} as follows. In an open neighborhood $\R^3\supset \Omega\supset\rSurf$ we can extend $f$ constantly into the ambient space $F:\Omega\to \R$. By computing the classical gradient in $\R^3$ of $F$ and projecting the result back to the tangent space yields the surface gradient of $f$
\begin{align}
	\rgrad f=\Proj\nabla F,\qquad \Proj=\bI-\rnv\otimes \rnv,
\end{align}
which is independent of the chosen extension $F$ \cite[Lemma 2.4]{DE13}. Here, $\Proj$ denotes the projection onto the tangent space of $\rSurf$. For a vector-valued function $u:\rSurf\to\R^3$ its surface gradient is defined component-wise as matrix, where each row is a gradient,
\begin{align}
	\rgrad u = \mat{\rgrad u_1^\top\\\rgrad u_2^\top\\\rgrad u_3^\top},\qquad \text{ such that }\qquad\rgrad u\rnv=0, \text{ but } \rgrad u^\top\rnv\neq 0.
\end{align}
The covariant derivative projects additionally the columns of $\rgrad u$ to the tangent space
\begin{align}
	\rcovder u = \Proj\rgrad u,\qquad \text{ such that }\qquad\rcovder u\rnv=\rcovder u^\top\rnv= 0.
\end{align}
The surface divergence of a vector-valued function $u:\rSurf\to \R^3$ is defined as $\div_{\rSurf}u:=\tr{\rgrad u}$ \cite{SF19a,Reus20} and the divergence of a matrix-valued function is given row-wise. This enables the integration-by-parts formula for continuously differentiable functions $u\in\Cone[\rSurf,\R^3]$ and $\bA\in\Cone[\rSurf,T\rSurf\otimes T\rSurf]$ 
\begin{align}
	\label{eq:ibp_tangent}
	\int_{\rSurf}\rgrad u:\bA\dx = -\int_{\rSurf} u\cdot\div_{\rSurf}\bA\dx+ \int_{\d\rSurf}u\cdot (\bA\rcnv)\ds,
\end{align}
where $\rcnv$ denotes the co-normal vector at the boundary. We will also make use of the Riemannian Hessian on $\rSurf$, which is defined for a scalar field $f\in \Ctwo[\rSurf]$ as \cite{SF19a,DE13}
\begin{align}
	\label{eq:surface_hesse}
	\rhesse f= \rcovder\rgrad f.
\end{align}

\subsection{Discrete and deformed surfaces}
\label{subsec:discrete_def_surfaces}
On $\rSurf$ we assume a shape regular triangulation $\T$ of (possibly curved) triangles and/or quadrilaterals, where the vertices of $\T$ lie on $\rSurf$. We emphasize that the discrete approximation $\T$ of $\rSurf$ is continuous but not necessarily $\Cone$. For each element $T\in\T$ the normal vector $\rnv$, which orientation is inherited from $\rSurf$, is given. On the element-boundaries, denoted by $\d T$, we define the tangent vector $\rtv_T$ oriented clock-wise such that the co-normal vector $\rcnv_T=\rnv_T\times \rtv_T$ points outward the element, cf. Figure~\ref{fig:nv_tv_cnv}. The set of all edges $E$ is denoted as skeleton $\E$ of $\T$. Further, on each edge we fix a tangential $\rtv_E$ and co-normal vector $\rcnv_E$ determining the orientation of $E$. Integrating over the triangulation $\T$ will be abbreviated by $\int_{\T}f\,\dx:=\sum_{T \in \T}\int_{T}f\,\dx$.

\begin{figure}[h]
	\centering
	\includegraphics[width=0.3\textwidth]{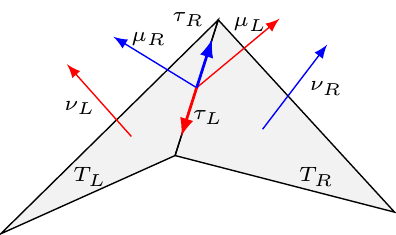}
	\caption{Elements with normal, edge tangential, and co-normal vectors.}
	\label{fig:nv_tv_cnv}
\end{figure}

Assume another surface $\dSurf\subset\R^3$ related to $\rSurf$ by a smooth function $\phi:\rSurf\to\dSurf$, which gradient $\rgrad\phi$ has rank two. Then, the derivative $\DGS=\rgrad\phi$ maps the tangent space $T\rSurf$ to $T\dSurf$ in a bijective manner. We will use the same symbol $\phi$ as mapping between the exact surfaces $\rSurf$, $\dSurf$ and the corresponding triangulations $\T$ and $\dT=\phi(\T)$. On the deformed triangulation $\dT$ we can define analogously element-wise the orthonormal frame $(\dnv,\dtv,\dcnv)$. We will need the transformation rule for the normal vector $\dnv$. Therefore, we use the cofactor matrix defined for invertible matrices by $\cof{\bA} = \det(\bA)\bA^{-\top}$. Note that the cofactor matrix is well-defined also for singular matrices. In three spatial dimensions the tensor cross product \cite{BGO16} is a beneficial tool. It is symmetric, bilinear, and there holds $\cof{\bA}=\frac{1}{2}\bA\tcross\bA$ (further identities are stated in Lemma~\ref{lem:tcp_id}). The variation of the cofactor matrix with respect to $\bA$ in direction $\delta\bA$ is then $D_A\cof{\bA}[\delta\bA] = \bA\tcross(\delta \bA)$. There holds on each element $T\in\T$ with $J:=\|\cof{\DGS}\rnv\|_2$
\begin{align}
	\label{eq:trafo_nv}
	\dnv= J^{-1}\cof{\DGS}\rnv.
\end{align}

\subsection{Shell models}
\label{subsec:shell_models}
We consider in this work the direct method for shells going back to the Cosserat brothers \cite{CC09}, also denoted as geometrically exact shell model \cite{SIMO89}. Following the approach of an inextensible one-director Cosserat surface the reference (initial) configuration of a shell $\re{\Omega}\subset \R^3$ with homogeneous thickness $t$ is given by
\begin{align}
	\re{\Omega} =\{ \rx + z\rnv(\rx)\,\vert\, \rx\in\rSurf,\, z\in [-t/2,t/2]\},
\end{align}
where $\rSurf$ denotes the mid-surface with its oriented unit normal vector $\rnv$. The deformed configuration of the shell $\de{\Omega}=\Phi(\re{
	\Omega})$ is defined by $\Phi:\re{\Omega}\to\R^3$ being of the form
\begin{align}
	\Phi(\rx,z) = \phi(\rx)+z\director(\rx),\qquad \rx\in\rSurf, z\in [-t/2,t/2].
\end{align}
Here, $\phi:\rSurf\to\dSurf$ is the deformation mapping the mid-surface of the initial configuration to its deformed counter-part and we define the displacement $u:\rSurf\to\R^3$ by $u=\phi-\idop$. The director is denoted by $\director:\rSurf\to \Stwo$. On the mid-surface the deformation gradient $\DGS$, Cauchy--Green strain $\CGTS$, and Green strain tensor $\GTS$ are given by $\DGS=\rgrad\phi=\Proj+\rgrad u$, $\CGTS=\DGS^\top\DGS$, and $\GTS=\frac{1}{2}(\CGTS-\Proj)=\frac{1}{2}(\rgrad u^\top\rgrad u + \rcovder u^\top + \rcovder u)$.

Following standard literature \cite{BRI17,Ciarlet05,SIMO89,CB11} the shell energy reads in TDC notation, assuming a linear material law of St-Venant-Kirchhoff, for given body force $f:\rSurf\to\R^3$
\begin{subequations}
	\begin{align}
		E_{\mathrm{shell}}^{\mathrm{Naghdi}} &= E_{\mathrm{mem}} + E_{\mathrm{bend}} + E_{\mathrm{shear}} - \int_{\rSurf}f\cdot u\dx\label{eq:Naghdi_shell_energy},\\
		E_{\mathrm{mem}}&=\frac{t}{2}\int_{\rSurf}\|\GTS\|_{\MT}^2\dx,\\	
		E_{\mathrm{bend}}&=\frac{t^3}{24}\int_{\rSurf}\|\sym{\DGS^\top\rgrad(\director)}-\rgrad\rnv\|_{\MT}^2\dx,\\
		E_{\mathrm{shear}}&=\frac{t\shearModulus\shearCorrection}{2}\int_{\rSurf}\|\DGS^\top\director\|_2^2\dx,\label{eq:shear_energy}
	\end{align}
\end{subequations}
where $E_{\mathrm{mem}}$, $E_{\mathrm{bend}}$, and $E_{\mathrm{shear}}$ are the membrane, bending, and shearing energies, respectively. Therein, $\kappa=\frac{5}{6}$ and $G=\frac{\YoungsModulus}{2(1+\PoissonRatio)}$ denote the shear correction factor and shearing modulus, and the fourth order material tensor $\MT$ involving the Young's modulus $\YoungsModulus$ and Poisson ratio $\PoissonRatio$ reads
\begin{align}
	\| \GTS\|_{\MT}^2 = \frac{\YoungsModulus}{1-\PoissonRatio^2}\big(\PoissonRatio\tr{\GTS}^2 + (1-\PoissonRatio)\GTS:\GTS\big),
\end{align}
where the plane stress assumption has already been incorporated.

Energy formulation \eqref{eq:Naghdi_shell_energy} corresponds to a 5-parameter Naghdi shell model with three displacement field components and two rotational degrees not including drilling rotations. Assuming the Kirchhoff--Love hypothesis, i.e., that the director $\director$ has to be perpendicular to the deformed mid-surface, $\director=\dnv\in\{T\dSurf\}^\perp$, we obtain the 3-parameter Koiter shell model including only displacement fields
\begin{subequations}
	\begin{align}
		E_{\mathrm{shell}}^{\mathrm{Koiter}} &= E_{\mathrm{mem}} + \tilde{E}_{\mathrm{bend}} - \int_{\rSurf}f\cdot u\dx\label{eq:Koiter_shell_energy},\\
		\tilde{E}_{\mathrm{bend}} &= \frac{t^3}{24}\int_{\rSurf}\|\DGS^\top\rgrad(\dnv)-\rgrad\rnv\|_{\MT}^2\dx.\label{eq:Koiter_bending_term}
	\end{align}
\end{subequations}
Note, that $\DGS^\top\dnv=0$ such that the shearing energy is zero in this setting.

\subsection{Function spaces}
\label{subsec:function_spaces}
In the methods described in the following sections the displacement field $u$, bending moment tensor $\bending$, shearing field $\rshear$ as well as a hybridization unknown $\angle$ will be used and discretized by suited finite elements. For a better readability and structure we postpone the definition of the specific finite element spaces to Section~\ref{sec:finite_elements} and first define function spaces on the triangulation $\T$ with appropriate continuity assumptions. Let $\Cinf[\T]:=\Pi_{T\in\T}\Cinf[T]$ denote the set of piece-wise smooth functions on $\T$, which do not require any continuity across elements. The set of piece-wise smooth vector-valued functions is given by $\Cinf[\T,\R^3]$ and analogously for matrix-valued functions. We start with the displacement field $u$ being globally continuous and therefore in the Sobolev space of weakly differentiable functions
\begin{align}
	\label{eq:displ_space}
	u\in \spaceU(\T):=\{ v\in \Cinf[\T,\R^3]\,\vert\, v \text{ is continuous} \}\subset\Hone[\T].
\end{align}
The bending moment tensor, used later as independent field, is symmetric, has values solely in the tangent space, and requires only its co-normal--co-normal component to be continuous over elements, $\E^{\mathrm{int}}:=\E\setminus\d \rSurf$,
\begin{align}
	\label{eq:bending_space}
	\bending\in \spaceB(\T):=\{ \bm{\tau}\in\Cinf[\T,T\rSurf\otimes T\rSurf]\,\vert\, \bm{\tau}^\top=\bm{\tau},\, \jump[E]{\bm{\tau}_{\rcnv\rcnv}}=0\,\forall E\in\E^{\mathrm{int}} \},
\end{align}
where $\bm{\tau}_{\rcnv\rcnv}:= \rcnv^\top \bm{\tau}\rcnv$ and $\jump[E]{\bm{\tau}_{\rcnv\rcnv}} = (\bm{\tau}_{\rcnv\rcnv})|_{T_L}-(\bm{\tau}_{\rcnv\rcnv})|_{T_R}$ denotes the co-normal--co-normal component and the jump over the edge $E=T_L\cap T_R$, respectively. The shearing field used for the 5-parameter Naghdi shell will be a tangential continuous vector field
\begin{align}
	\label{eq:shearing_space}
	\rshear\in \spaceS(\T) :=\{ \beta\in \Cinf[\T,T\rSurf]\,\vert\, \jump[E]{\beta_{\rtv}}=0\,\forall E\in\E^{\mathrm{int}} \}\subset\HCurl[\T],
\end{align}
where $\beta_{\rtv}:=\beta\cdot \rtv$ denotes the edge tangential component of $\beta$, and $\jump[E]{\beta_{\rtv}} = (\beta_{\rtv})|_{T_L}+(\beta_{\rtv})|_{T_R}=(\beta|_{T_R}-\beta|_{T_R})\cdot\rtv_E$. The function space $\HCurl[\T]$ consists of square-integrable vector fields on $\rSurf$ which surface curl is also square-integrable. We will use a hybridization field $\angle$ for which we define the following space living solely on the skeleton $\E$. We equip an $f\in\Cinf[\E]$ with the co-normal vector $\rcnv$, but, having from both element sides the same orientation as $\rcnv_E$. With $\mathrm{sgn}:\R\to\{-1,0,1\}$ denoting the sign function
\begin{align}
	\label{eq:hyb_space}
	\angle\in \spaceH(\T):=\{ f\rcnv\,\vert\, f\in\Cinf[\E] \wedge \mathrm{sgn}(\rcnv|_{T_L}\cdot\rcnv_E)=\mathrm{sgn}(\rcnv|_{T_R}\cdot\rcnv_E) \forall E=T_L\cap T_R\}.
\end{align}
Note, that for $\angle\in \spaceH(\T)$ there holds $\jump[E]{\angle_{\rcnv}}=f-f=0$ as well as $\angle\cdot \rtv_E=(\angle\cdot \rnv)|_{T_L}=(\angle\cdot \rnv)|_{T_R}=0$. Besides the jump $\jump{\cdot}$ we additionally define the mean at an edge $E=T_L\cap T_R$ by $\mean{f}:= \frac{1}{2}(f|_{T_L}+f|_{T_R})$.

\section{HHJ for nonlinear Koiter shells}
\label{sec:HHJ_nonlinear_Koiter_shells}
For the reader's convenience and for completeness we motivate and derive the HHJ method for nonlinear Koiter shells proposed in \cite{NS19}, however, in a more general procedure starting with a three-field formulation.

Assuming the Kirchhoff--Love hypothesis, i.e., that normal vectors of the initial configuration remain perpendicular to the deformed configuration of the mid-surface, leads to zero shearing energy, $E_{\mathrm{shear}}=0$. A fourth order problem due to the bending energy \eqref{eq:Koiter_bending_term} is obtained
\begin{align}
	\label{eq:bend_energy_fourth_order}
	\tilde{E}_{\mathrm{bend}}=\frac{t^3}{24}\int_{\T}\|\DGS^\top\rgrad\Big(\frac{\cof{\DGS}\rnv}{J}\Big)-\rgrad \rnv\|_{\MT}^2\dx,
\end{align}
as the pull-back $\dnv = \frac{1}{J}\cof{\DGS}\rnv$ \eqref{eq:trafo_nv} already involves one derivative. 

\subsection{Distributional curvature}
\begin{figure}[h]
	\centering
	\includegraphics[width=0.3\textwidth]{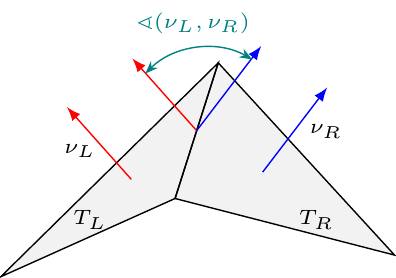}
	\caption{Dihedral angle $\sphericalangle(\nu_L,\nu_R)$ between two elements.}
	\label{fig:dihedral_angle}
\end{figure}
An additional problem of energy \eqref{eq:bend_energy_fourth_order} is that for an affine triangulation $\T$ the normal vector $\rnv$ is piece-wise constant and thus zero bending would be measured, $\rgrad\rnv|_T=0$. The whole information about curvature of an affine triangulation is concentrated in the change of the dihedral angle on the edges between two elements, cf. Figure~\ref{fig:dihedral_angle}. In \cite{NSS23,Neun21} it was shown that the distributional Weingarten tensor $\rgrad\rnv$ acts on co-normal--co-normal continuous functions and reads
\begin{align}
	\langle\rgrad\rnv,\bm{\Psi}\rangle = \sum_{T \in \T}\int_{T}\rgrad\rnv:\bm{\Psi}\,\dx + \sum_{E\in\E}\int_E\sphericalangle(\rnv_L,\rnv_R)\bm{\Psi}_{\rcnv\rcnv}\,\ds,\quad \forall \bm{\Psi}\in\spaceB(\T).
\end{align}
For an affine triangulation the element terms are zero and only the angle remains. This approach of approximating the curvature by means of angles is related to Steiner's offset formula \cite{Stein40} and well-known in the field of discrete differential geometry, see e.g. \cite{GGRZ06} and therein references.

\subsection{Three-field and two-field mixed formulation of HHJ}
To incorporate the distributional curvature in the bending energy a lifting of the distributional curvature difference to a regular function is performed. Therefore, the curvature difference $\curv^\mathrm{diff}$ of the initial and deformed configuration is introduced together with the bending moment $\bending$ as energetic conjugate acting as Lagrange multiplier. The Lagrangian of the corresponding Hu--Washizu three-field bending energy is
\begin{align}
	\L_{\mathrm{bend}}(u,\curv^\mathrm{diff},\bending)=&\int_{\T}\frac{t^3}{24}\|\curv^\mathrm{diff}\|^2_{\MT} + \big((\DGS^\top\rgrad\big(\frac{\cof{\DGS}\rnv}{J}\big)-\rgrad\rnv)-\curv^\mathrm{diff}\big):\bending\,\dx\nonumber \\
	&+ \sum_{E\in\E}\int_E (\sphericalangle(\nu_L,\nu_R)\circ\phi-\sphericalangle(\rnv_L,\rnv_R))\bending_{\rcnv\rcnv}\,\ds.\label{eq:HHJ_bend_three_field}
\end{align}

Note, that the bending moment tensor $\bending$ is symmetric, lives only on the tangent space of the mid-surface, $\bending\in T\rSurf\otimes T\rSurf$, and requires continuous co-normal--co-normal component. Therefore, we use $\bending\in\spaceB(\T)$, cf. \eqref{eq:bending_space}.

For a linear material law, represented through $\MT$, we can eliminate $\curv^{\mathrm{diff}}$ leading to a Hellinger--Reissner mixed two-field formulation.
\begin{lemma}
	If $\MT$ is invertible \eqref{eq:HHJ_bend_three_field} is equivalent to
	\begin{align}
		\label{eq:first_HHJ_method}
		\begin{split}
			\L_{\mathrm{bend}}(u,\bending)= -\frac{6}{t^3}&\int_{\T}\|\bending\|^2_{\MT^{-1}}\,\dx+\int_{\T} (\DGS^\top\rgrad(\dnv)-\rgrad\rnv):\bending\dx \\
			&+ \sum_{E\in\E} \int_E (\sphericalangle(\nu_L,\nu_R)\circ\phi-\sphericalangle(\rnv_L,\rnv_R))\bending_{\rcnv\rcnv}\ds,
		\end{split}
	\end{align}
	where the inversion of the material law reads\footnote{A wrong factor of 2 in \cite[(2.20)]{NS19} has been corrected here.}
	\begin{align}
		\|\bending\|^2_{\MT^{-1}}= \frac{1+\PoissonRatio}{\YoungsModulus}\big(\bending:\bending-\frac{\PoissonRatio}{\PoissonRatio+1}\tr{\bending}^2\big).\label{eq:inverted_material_law}
	\end{align}
\end{lemma}
\begin{proof}
	We compute the variation of \eqref{eq:HHJ_bend_three_field} with respect to $\curv^{\mathrm{diff}}$
	\begin{align*}
		D_{\curv^{\mathrm{diff}}}\L(u,\curv^{\mathrm{diff}},\bending)[\delta \curv]=\frac{t^3}{12}\int_{\T}(\MT\curv^{\mathrm{diff}}):\delta \curv-\bending:\delta\curv\,\dx\overset{!}{=}0\forall\delta\curv,
	\end{align*}
	from which we extract $\curv^{\mathrm{diff}}=\frac{12}{t^3}\MT^{-1}\bending$. Inserting into the Lagrange functional we get
	\begin{align*}
		\L(u,\bending)= \int_{\T}\frac{6}{t^3}\|\bending\|_{\MT^{-1}}^2&-\frac{12}{t^3}\|\bending\|^2_{\MT^{-1}}\,\dx+ \int_{\T}(\DGS^\top\rgrad\Big(\frac{\cof{\DGS}\rnv}{J}\Big)-\rgrad\rnv):\bending\,\dx\nonumber \\
		&+ \sum_{E\in\E}\int_E (\sphericalangle(\nu_L,\nu_R)\circ\phi-\sphericalangle(\rnv_L,\rnv_R))\bending_{\rcnv\rcnv}\,\ds.
	\end{align*}
	The claim follows by simplifying the first two terms.
\end{proof}
The volume term of the bending formulation can be rewritten as follows.
\begin{lemma}
	\label{lem:Koiter_bending_volterm}
	There holds for all $T\in\T$
	\begin{align}
		\int_T(\DGS^\top\rgrad(\dnv)-\rgrad\rnv):\bending\dx=-\int_T (\Hessian_{\dnv}+(1-\rnv\cdot\dnv)\rgrad\rnv):\bending\dx,
	\end{align}
	where $\Hessian_{\dnv}=\sum_{i=1}^3\rhesse u_i\dnv[i]$ with $\rhesse u_i$ the surface Hessian \eqref{eq:surface_hesse}.
\end{lemma}
\begin{proof}
	See \cite[(2.15) and Appendix A]{NS19}.
\end{proof}
Therefore, the two-field HHJ method for nonlinear Koiter shells reads
\begin{align}
	\L^{\mathrm{HHJ}}(u,\bending)&= E_{\mathrm{mem}}(u)-\frac{6}{t^3}\int_{\T}\|\bending\|^2_{\MT^{-1}}\,\dx-\int_{\T} (\Hessian_{\dnv}+(1-\rnv\cdot\dnv)\rgrad\rnv):\bending\dx \nonumber\\
	&\quad+ \sum_{E\in\E} \int_E (\sphericalangle(\nu_L,\nu_R)\circ\phi-\sphericalangle(\rnv_L,\rnv_R))\bending_{\rcnv\rcnv}\ds.\label{eq:HHJ_nonlinear}
\end{align}
The approach entails two important advantages:
\begin{enumerate}
	\item the fourth order problem is reduced to a second order mixed saddle point problem,
	\item the distributional curvature including the angle difference at the edges is directly incorporated into the formulation such that also for affine triangulations the correct change of curvature is measured.
\end{enumerate}

\subsection{Hybridization}
\label{subsec:hybridization}
The possible disadvantage of having a saddle point instead of a minimization problem can be overcome by hybridization techniques \cite{BBF13}. To this end the co-normal--co-normal continuity of $\bending\in\spaceB(\T)$ is broken, making $\bending$ completely discontinuous, denoted by $\bending\in \spaceB^{\mathrm{dc}}(\T)$. Then the continuity gets reinforced weakly in the formulation by means of a Lagrange multiplier $\angle\in\spaceH(\T)$, see \eqref{eq:hyb_space}, which lives solely on the skeleton $\E$ and is co-normal continuous, $\jmp{\angle_{\rcnv}}=0$,
\begin{align}
	\begin{split}
		\L^{\mathrm{hyb}}_{\mathrm{bend}}(u,\bending,\angle)= &-\int_{\T}\frac{6}{t^3}\|\bending\|^2_{\MT^{-1}}\,\dx+\int_{\T} (\DGS^\top\rgrad(\dnv)-\rgrad\rnv):\bending\dx\\ &+\sum_{E\in\E}\int_E(\sphericalangle(\nu_L,\nu_R)\circ\phi-\sphericalangle(\rnv_L,\rnv_R))\mean{\bending_{\rcnv\rcnv}}+ \angle_{\rcnv}\jmp{\bending_{\rcnv\rcnv}}\ds.
	\end{split}
\end{align}
The physical meaning of $\angle$ will be discussed in Section~\ref{sec:physical_meaning_hyb} as it depends on the specific implementation of the angle difference. After performing static condensation the dofs of $\bending$ are eliminated at element level and the remaining global system in $(u,\angle)$ corresponds to a minimization problem again. For a general description of the condensation procedure, we refer to e.g. \cite{Hugh00,BBF13} and for the TDNNS method \cite[Section 5.4]{Neun21}. The hybridized version of \eqref{eq:HHJ_nonlinear} is given by
\begin{align}
	\L_{\mathrm{hyb}}^{\mathrm{HHJ}}(u,\bending,\angle)&= E_{\mathrm{mem}}(u)-\frac{6}{t^3}\int_{\T}\|\bending\|^2_{\MT^{-1}}\,\dx-\int_{\T} (\Hessian_{\dnv}+(1-\rnv\cdot\dnv)\rgrad\rnv):\bending\dx \nonumber \\
	&\quad+ \sum_{E\in\E}\int_E(\sphericalangle(\nu_L,\nu_R)\circ\phi-\sphericalangle(\rnv_L,\rnv_R))\mean{\bending_{\rcnv\rcnv}}+ \angle_{\rcnv}\jmp{\bending_{\rcnv\rcnv}}\ds.\label{eq:HHJ_hyb_nonlinear}
\end{align}
As we will discuss in Section~\ref{sec:computational aspects} the current form of angle difference computation leads to numerical instabilities and present therein a stable algorithm.

\section{The TDNNS method for nonlinear Naghdi shells}
\label{sec:TDNNS_nonlinear_Naghdi_shells}
For Naghdi shells  shearing/rotational dofs need to be added. Although it is possible to obtain a second order bending energy term by using the rotation in the bending, we use the same structure as in the previous section and add shearing dofs in a hierarchical way described in the following. The director $\director$ can be defined in the forms \cite{OSRB2017}
\begin{figure}
	\centering
	\includegraphics[width=0.55\textwidth]{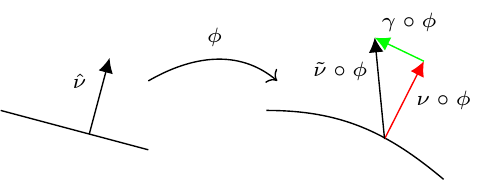}
	\caption{Hierarchical approach for director $\director$. $\dnv$ and $\dshear$ are the normal and shearing vectors on the deformed configuration.}
	\label{fig:hierarchical_director}
\end{figure}
\begin{align}
	\label{eq:director_def}
	\director=\tilde{\bR}(\dshear)\frac{\cof{\DGS}\rnv}{\|\cof{\DGS}\rnv\|} \qquad \mathrm{ or }\qquad\director=\frac{\cof{\DGS}\rnv+\dshear}{\|\cof{\DGS}\rnv+\dshear\|},
\end{align}
where $\tilde{\bR}(\dshear)\in \SO{3}$ denotes a rotation matrix depending on two angle parameters stored in $\dshear$, or $\dshear\in T\dSurf$ denoting a shear vector in the tangent space of the deformed configuration, cf. Figure~\ref{fig:hierarchical_director}. In both versions $\dshear$ appears nonlinearly guaranteeing that the director $\director$ has unit length, $\|\director\|=1$. Possible direct discretizations of the space of rotations $\SO{3}$ are given by e.g. geometric finite elements (GFEs) \cite{HS2020,NSBN2023} or using the Rodrigues rotation vector \cite{CMPS2021,SPI2023}. In this work, however, we consider a different approach. As discussed in \cite{OSRB2017} the shearing parameter turns out to be small in realistic shell examples, $\dshear=\mathcal{O}(\veps)$, $\varepsilon\ll 1$, even for large deformations. This motivates to linearize \eqref{eq:director_def}. As the tangent space at the identity of $\SO{3}$ is the set of skew symmetric matrices there holds 
\begin{align}
	\label{eq:lin_director}
	\director=\tilde{\bR}(\dshear)\dnv \approx (\bI+\skew{\gamma_1,\gamma_2})\dnv = \dnv + \dshear,
\end{align}
where $\gamma_1,\gamma_2$ define the two rotation parameters and we defined the shear $\dshear:=\skew{\gamma_1,\gamma_2}\dnv \in T\dSurf=\{\dnv\}^\perp$. Therefore, an additive splitting of the director into the deformed normal and the shearing is obtained. As $\dshear$ lies in the tangent space of the deformed configuration we can use the classical push forward $\DGS$ or the covariant mapping, performed by the Moore--Penrose pseudo inverse, for the transformation
\begin{align}
	\label{eq:trafo_shear}
	\dshear=\DGS\rshear\qquad \mathrm{ or }\qquad \dshear=\DGS^{\dagger^\top}\rshear.
\end{align}
We consider the second option of the covariant transformation as we are going to discretize $\rshear$ with $\HCurl$-conforming tangential continuous N\'ed\'elec finite elements, described in Section~\ref{sec:finite_elements}, i.e. $\rshear\in\spaceS(\T)$ \eqref{eq:shearing_space}. These finite elements correspond to Whitney forms \cite{whitney57} discretizing so-called one-forms and the covariant transformation is the natural choice preserving the tangential continuity.

In analogy to the HHJ method in the previous section we can rewrite the bending element term.
\begin{lemma}
	There holds for all $T\in\T$
	\begin{align}
		\int_T(\sym{\DGS^\top\rgrad(\director)}-\rgrad\rnv):\bending\dx =-\int_T (\Hessian_{\director}+(1-\rnv\cdot\director)\rgrad\rnv-\rgrad\rshear):\bending\dx.\label{eq:volume_term_rewritten}
	\end{align}
\end{lemma}
\begin{proof}
	As $\bending$ is symmetric we can neglect the symmetrization $\sym{\cdot}$. Further, as $\director=\dnv+\dshear$ we use Lemma~\ref{lem:Koiter_bending_volterm} for $\dnv$ and only need to consider the shearing part. Following the structure of \cite[Appendix A]{NS19} we compute with integration-by-parts \eqref{eq:ibp_tangent} and the product rule
	\begin{align*}
		\int_{T}\DGS^\top\rgrad(\dshear):&\bending\dx =\int_T \rgrad(\DGS^{\dagger^\top}\rshear):\DGS\bending\dx \\
		&= -\int_T \DGS^{\dagger^\top}\rshear\cdot(\mat{\bH_1:\bending\\\bH_2:\bending\\\bH_3:\bending}+\DGS\div_{\rSurf}\bending)\dx+\int_{\d T}\DGS^{\dagger^\top}\rshear\cdot\DGS\bending_{\rcnv}\ds,
	\end{align*}
	where $\bH_i = \rgrad(\Proj_i)+\rhesse u_i$, $\Proj_i$ denoting the $i$th column of $\Proj$. Further
	\begin{align*}
		\sum_{i=1}^3\dshear[i](\rgrad(\Proj_i)+\rhesse u_i):\bending
		&=\sum_{i=1}^3\dshear[i]((\rhesse u_i-\rgrad\rnv_i\otimes\rnv-\rnv_i\rgrad\rnv):\bending\\
		&=(\Hessian_{\dshear}-(\dshear\cdot\rnv)\rgrad\rnv):\bending
	\end{align*}
	and thus with $\DGS^\dagger\DGS = \Proj$ and integration-by-parts
	\begin{align*}
		\int_{T}\DGS^\top\rgrad(\dshear):\bending\dx
		&= -\int_T (\Hessian_{\dshear}-(\dshear\cdot\rnv)\rgrad\rnv-\rgrad\rshear):\bending\dx.
	\end{align*}
	Combining with Lemma~\ref{lem:Koiter_bending_volterm} finishes the proof.
\end{proof}
The additive splitting \eqref{eq:lin_director} resembles also in the boundary part of the bending energy. From the TDNNS method \cite{PS11,PS17} we obtain that the appropriate term is the co-normal components of the shear, being part of its distributional derivative
\begin{align}
	\int_{E}(\sphericalangle(\nu_L,\nu_R)\circ\phi-\sphericalangle\left(\rnv_L,\rnv_R\right)-\jmp{\rshear_{\rcnv}})\bending_{\rcnv\rcnv}\,ds.
\end{align}

For the shearing energy \eqref{eq:shear_energy} we directly deduce with $\DGS^\top\dnv =0$
\begin{align}
	\|\DGS^\top\director\|^2 = \|\DGS^\top(\dnv+\dshear)\|^2=\|\rshear\|^2.
\end{align}
Therefore, the (hybridized) TDNNS method for nonlinear Naghdi shells reads
\begin{subequations}
	\label{eq:TDNNS_nonlinear_Naghdi_shell}
	\begin{align}
		&\L^{\mathrm{TDNNS}}(u,\bending,\rshear)= E_{\mathrm{mem}}(u)-\int_{\T}\frac{6}{t^3}\|\bending\|^2_{\MT^{-1}}+ (\Hessian_{\director}+(1-\rnv\cdot\director)\rgrad\rnv-\rgrad\rshear):\bending\dx \nonumber\\
		&+\frac{t\kappa G}{2}\int_{\T}\|\rshear\|^2\dx+ \sum_{E\in\E} \int_E (\sphericalangle(\nu_L,\nu_R)\circ\phi-\sphericalangle(\rnv_L,\rnv_R)-\jmp{\rshear_{\rcnv}})\bending_{\rcnv\rcnv}\ds,\label{eq:TDNNS_nonlinear}\\
		&\L_{\mathrm{hyb}}^{\mathrm{TDNNS}}(u,\bending,\rshear,\angle)= E_{\mathrm{mem}}(u)-\int_{\T}\frac{6}{t^3}\|\bending\|^2_{\MT^{-1}}+ (\Hessian_{\director}+(1-\rnv\cdot\director)\rgrad\rnv-\rgrad\rshear):\bending\dx \nonumber \\
		&+\frac{t\kappa G}{2}\int_{\T}\|\rshear\|^2\dx+  \sum_{E\in\E}\int_E(\sphericalangle(\nu_L,\nu_R)\circ\phi-\sphericalangle(\rnv_L,\rnv_R))\mean{\bending_{\rcnv\rcnv}}-\jmp{\rshear_{\rcnv}\bending_{\rcnv\rcnv}}+ \angle_{\rcnv}\jmp{\bending_{\rcnv\rcnv}}\ds.\label{eq:TDNNS_hyb_nonlinear}
	\end{align}
\end{subequations}
We will justify that \eqref{eq:TDNNS_nonlinear_Naghdi_shell} extends the TDNNS method from linear Reissner--Mindlin plates \cite{PS17} to nonlinear Naghdi shells in the following section.

\section{Linearization to linear Kirchhoff--Love and Reissner--Mindlin plates and shells}
\label{sec:linearization}
In Sections~\ref{sec:HHJ_nonlinear_Koiter_shells} and \ref{sec:TDNNS_nonlinear_Naghdi_shells} we proposed the HHJ and TDNNS method for nonlinear Koiter and Naghdi shells, respectively. We now linearize these two models first in the setting of shells and then for plates justifying that the presented methods extend the HHJ and TDNNS method from linear plates to nonlinear shells. For ease of notation we consider the formulations without hybridization variable $\angle$ and emphasize that the linearization procedure can be applied in the same manner.

We start with the element terms and then focus on the non-standard edge terms.
\begin{lemma}
	\label{lem:linearization_volume_terms}
	Let $\rgrad u=\mathcal{O}(\veps)$, $\rshear=\mathcal{O}(\veps)$, and $\rgrad\rshear=\mathcal{O}(\veps)$. Then there holds
	\begin{subequations}
		\begin{align}
			\director&= \rnv - \rgrad u^\top\rnv + \rshear+\mathcal{O}(\veps^2),\label{eq:lin_nv}\\
			\GTS(u) &=\sym{\rcovder u}+\mathcal{O}(\veps^2),\label{eq:lin_membrane}\\
			-(\Hessian_{\director}+(1-\director\cdot\rnv)\rgrad\rnv-\rgrad\rshear) &= -(\Hessian_{\rnv}-\rgrad\rshear)+\mathcal{O}(\veps^2) \label{eq:lin_bend}.
		\end{align}
	\end{subequations}
\end{lemma}
\begin{proof}
	We use $\dnv=\frac{\frac{1}{2}\DGS\tcross\DGS\rnv}{\|\frac{1}{2}\DGS\tcross\DGS\rnv\|}$, \eqref{eq:tcp_id5}, \eqref{eq:tcp_id6}, and Taylor at $\nabla u=0$
	\begin{align*}
		\dnv &= \rnv + \frac{\Proj\tcross\rgrad u \rnv}{\|\frac{1}{2}\Proj\tcross\Proj \rnv\|}-\frac{\frac{1}{2}\Proj\tcross\Proj \rnv}{\|\frac{1}{2}\Proj\tcross\Proj \rnv\|^3}(\Proj\tcross\rgrad u \rnv)\cdot(\frac{1}{2}\Proj\tcross\Proj \rnv)+\mathcal{O}(\veps^2)\\
		&=(\bI\tcross\rgrad u)\rnv-(\rnv\cdot(\bI\tcross\rgrad u\rnv))\rnv+\mathcal{O}(\veps^2).
	\end{align*}
	Further, with \eqref{eq:tcp_id3} we find $\bI\tcross\rgrad u=\div_{\rSurf} u\,\bI - \rgrad u^\top$ and thus \eqref{eq:lin_nv} follows
	\begin{align*}
		\dnv&= \rnv + \div_{\rSurf}u\,\rnv-\rgrad u^\top \rnv - \div_{\rSurf}u\,\rnv +\mathcal{O}(\veps^2)= \rnv -\rgrad u^\top \rnv +\mathcal{O}(\veps^2).
	\end{align*}
	Inserting the definition of the Green-strain tensor directly yields \eqref{eq:lin_membrane}
	\begin{align*}
		\GTS(u) = \frac{1}{2}\big(\Proj\rgrad u+\rgrad u^\top\Proj \big) +\mathcal{O}(\veps^2) = \sym{\rcovder u}+\mathcal{O}(\veps^2).
	\end{align*}
	For \eqref{eq:lin_bend} we use \eqref{eq:volume_term_rewritten} and \eqref{eq:lin_nv}
	\begin{align*}
		(\DGS^\top\rgrad(\dnv)-\rgrad\rnv):\bending&= (\rcovder(\rshear-\rgrad u^\top \rnv) + \rgrad u^\top\rgrad(\rnv-\rgrad u^\top \rnv+\rshear)):\bending+\mathcal{O}(\veps^2)\\
		&=(-\rcovder(\rgrad u^\top\rnv)+\rgrad\rshear +\rgrad u^\top\rgrad\rnv):\bending+\mathcal{O}(\veps^2)\\
		&=-(\Hessian_{\rnv}-\rgrad\rshear):\bending+\mathcal{O}(\veps^2)
	\end{align*}
	proofing the claim.
\end{proof}
By setting $\rshear=0$ we directly obtain the linearizations of the Koiter shell due to the used hierarchical approach. Next, we linearize the difference of angles on the edge terms. As $\rshear$ already appears linearly we can omit it for the derivation.
\begin{lemma}
	\label{lem:linearization_angle_terms}
	Under the assumptions of Lemma~\ref{lem:linearization_volume_terms} there holds
	\begin{align}
		\sphericalangle(\nu_L,\nu_R)\circ\phi-\sphericalangle\left(\rnv_L,\rnv_R\right) = \jmp{(\rgrad u^\top\rnv)_{\rcnv}}+\mathcal{O}(\veps^2).
	\end{align}
\end{lemma}
\begin{proof}
	With Taylor and the identity 
	\begin{align*}
		\sqrt{1-(\rnv_L\cdot\rnv_R)^2}&=\cos(\frac{\pi}{2}-\arccos(\rnv_L\cdot\rnv_R))=\cos(\arccos(\rnv_R\cdot\rcnv_L))=\rnv_R\cdot\rcnv_L=\rnv_L\cdot\rcnv_R
	\end{align*}
	we obtain using \eqref{eq:tcp_id5} and \eqref{eq:tcp_id3}
	\begin{align*}
		\sphericalangle(\nu_L,\nu_R)\circ\phi-\sphericalangle\left(\rnv_L,\rnv_R\right) &=-\frac{1}{\rnv_L\cdot\rcnv_R}\Big( (\bI\tcross\rgrad u_L\rnv_L)\cdot\rnv_R+\rnv_L\cdot(\bI\tcross\rgrad u_R\rnv_R) \\
		&- \rnv_L\cdot\rnv_R((\bI \tcross \rgrad u_L\rnv_L)\cdot\rnv_L+(\bI \tcross \rgrad u_R\rnv_R)\cdot\rnv_R)  \Big)+\mathcal{O}(\veps^2)\\\
		&=-\frac{1}{\rnv_L\cdot\rcnv_R}\Big( (\div_{\rSurf}\rgrad u_L+\div_{\rSurf}\rgrad u_R)\rnv_L\cdot\rnv_R-(\rgrad u_L^\top\rnv_L)\cdot\rnv_R\\
		&-(\rgrad u_R^\top\rnv_R)\cdot\rnv_L - \rnv_L\cdot\rnv_R(\div_{\rSurf}\rgrad u_L+\div_{\rSurf}\rgrad u_R)  \Big)+\mathcal{O}(\veps^2).
	\end{align*}
	Using the splitting $\rnv_R=(\rnv_R\cdot\rnv_L)\rnv_L + (\rnv_R\cdot\rcnv_L)\rcnv_L$ (and analogously for $\rnv_L$) yields
	\begin{align*}
		\sphericalangle(\nu_L,\nu_R)\circ\phi-\sphericalangle\left(\rnv_L,\rnv_R\right)&=-\frac{1}{\rnv_L\cdot\rcnv_R}\Big(\frac{\d u_L}{\d\rcnv_L}\cdot\rnv_L(\rnv_R\cdot\rcnv_L)+\frac{\d u_R}{\d\rcnv_R}\cdot\rnv_R(\rnv_L\cdot\rcnv_R) \Big)+\mathcal{O}(\veps^2)\\
		&=\frac{\d u_L}{\d\rcnv_L}\rnv_L+\frac{\d u_R}{\d\rcnv_R}\rnv_R+\mathcal{O}(\veps^2)
	\end{align*}
	finishing the proof.
\end{proof}
Lemma~\ref{lem:linearization_angle_terms} is more general than the results in \cite{NS19} and \cite[Theorem 7.30]{Neun21}, where only plates have been considered or additionally small angles were assumed, respectively.

As a result, the HHJ and TDNNS method for linear Kirchhoff--Love and Reissner--Mindlin shells reads
\begin{align}
	\L^{\mathrm{HHJ}}_{\mathrm{lin}}(u,\bending)&= \frac{t}{2}\int_{\T}\|\sym{\rcovder u}\|_{\MT}^2\dx-\frac{6}{t^3}\int_{\T}\|\bending\|^2_{\MT^{-1}}\,\dx+\int_{\T} \Hessian_{\rnv}:\bending\dx \nonumber\\
	&\quad+ \sum_{E\in\E} \int_E \jmp{(\rgrad u^\top\rnv)_{\rcnv}}\bending_{\rcnv\rcnv}\ds,\label{eq:HHJ_linear_shell}\\
	\L^{\mathrm{TDNNS}}_{\mathrm{lin}}(u,\bending,\rshear) &= \frac{t}{2}\int_{\T}\|\sym{\rcovder u}\|_{\MT}^2\dx + \frac{t\kappa G}{2}\int_{\T}\|\rshear\|_2^2\dx -\frac{6}{t^3}\int_{\T}\|\bending\|^2_{\MT^{-1}}\,\dx\nonumber\\
	&\quad+\int_{\T}(\Hessian_{\rnv}-\rgrad\rshear):\bending\dx+\sum_{E\in\E} \int_E \jmp{(\rgrad u^\top\rnv)_{\rcnv}-\rshear_{\rcnv}}\bending_{\rcnv\rcnv}\ds.\label{eq:TDNNS_linear_shell}
\end{align}
As a next step we assume that the initial configuration of $\rSurf$ is a flat plate in the x-y plane such that $\rnv = \mat{0&0&1}^\top$. Then the membrane part decouples with the bending and shearing terms. By defining the vertical deflection $w:=u\cdot\rnv=u_3$ one directly obtains the HHJ \cite{Hel67,Her67,Joh73,Com89} and TDNNS \cite{PS17} method for plates
\begin{align}
	\L^{\mathrm{HHJ}}_{\mathrm{plate}}(u,\bending)&= -\frac{6}{t^3}\int_{\T}\|\bending\|^2_{\MT^{-1}}\,\dx+\int_{\T} \nabla^2 w:\bending\dx+ \sum_{E\in\E} \int_E \jmp{\frac{\d w}{\d\rcnv}}\bending_{\rcnv\rcnv}\ds,\label{eq:HHJ_plate}\\
	\L^{\mathrm{TDNNS}}_{\mathrm{plate}}(u,\bending,\rshear) &= \frac{t\kappa G}{2}\int_{\T}\|\rshear\|_2^2\dx -\frac{6}{t^3}\int_{\T}\|\bending\|^2_{\MT^{-1}}\,\dx\nonumber\\
	&\quad+\int_{\T}(\nabla^2 w-\nabla\rshear):\bending\dx+\sum_{E\in\E} \int_E \jmp{\frac{\d w}{\d\rcnv}-\rshear_{\rcnv}}\bending_{\rcnv\rcnv}\ds.\label{eq:TDNNS_plate}
\end{align}

\begin{remark}
	By performing the change of variables $\rshear = \rgrad w+\beta$ going from shearing to rotational dofs reveals the classical shear energy term for plates. Note that the change of variables $\rshear = \rgrad w+\beta$ holds for $w\in \Hone[\T]$ and $\rshear,\beta\in\HCurl[\T]$ in the continuous and discrete level exactly as $\nabla w\in\HCurl[\T]$. \eqref{eq:HHJ_plate} and \eqref{eq:TDNNS_plate} show the strong relationship between the HHJ method proposed in the 1960s and the TDNNS method developed in the last 15 years. Further, one directly obtains that in the limit of vanishing thickness, $t\to  0$, the solution of the TDNNS method converges to the solution of the HHJ method. This property is also extended into the case of linear and nonlinear shells, which can be readily checked from \eqref{eq:HHJ_linear_shell}, \eqref{eq:TDNNS_linear_shell} and \eqref{eq:HHJ_nonlinear}, \eqref{eq:TDNNS_nonlinear}, respectively. The hybridization techniques can directly be applied in the linearized cases. Then the resulting systems $(u,\angle)$ as well as $(u,\rshear,\angle)$ are minimization problems again, yielding a symmetric and positive definite stiffness matrix.
\end{remark}

\section{Computational aspects}
\label{sec:computational aspects}
In this section we treat several important computational aspects for implementing the above presented methods.

\subsection{Stable angle computation}
\label{subsec:stable_angle_comp}
For solving the nonlinear shell equations we need to compute the first and second variation of the angle computation. We expect for a triangulation $\T$ approximating a smooth surface $\rSurf$ that $\dnv[L]\cdot\dnv[R]\to 1$ for a sequence of triangulations. As the derivative of $\arccos(x)$, $\arccos^\prime(x)=\frac{-1}{\sqrt{1-x^2}}$, has a singularity at $x=\pm 1$, e.g. formulation \eqref{eq:HHJ_nonlinear} is numerically unstable. Therefore, we rewrite it into an algebraic equivalent formulation by using the averaged normal vector
\begin{align}
	\Av{\nu}\circ\phi = \frac{\dnv[L]+\dnv[R]}{\|\dnv[L]+\dnv[R]\|_2}
\end{align}
leading to (neglecting $\bending_{\rcnv\rcnv}$ for ease of presentation)
\begin{subequations}
	\label{eq:angle_rewritten}
	\begin{align}
		\sum_{E\in\E}\int_{E}\sphericalangle(\nu_L, \nu_R)\circ\phi-\sphericalangle(\rnv_L, \rnv_R)\ds
		&=\sum_{T \in \T}\int_{\partial T}\sphericalangle(\Av{\nu},\nu)\circ\phi-\sphericalangle(\Av{\rnv},\rnv)\ds\label{eq:angle_rewritten_a}\\
		&=\sum_{T\in\T}\int_{\partial T}\sphericalangle(\Av{\rnv}, \rcnv)-\sphericalangle(\Av{\nu}, \mu)\circ\phi\ds.\label{eq:angle_rewritten_b}
	\end{align}
\end{subequations}
\begin{figure}[h]
	\centering
	\begin{tabular}{ccc}
		\includegraphics[width=0.3\textwidth]{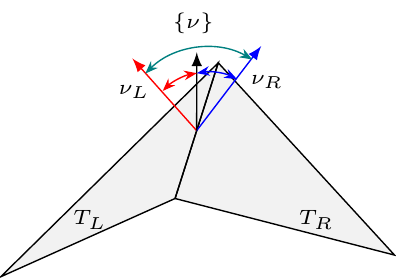} &\includegraphics[width=0.3\textwidth]{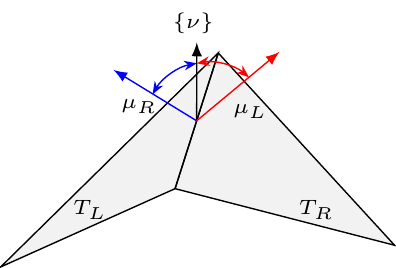} &\includegraphics[width=0.3\textwidth]{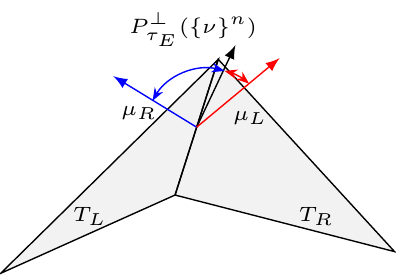} \\
		(a) & (b) & (c)
	\end{tabular}
	\caption{Angle computation. (a): With and without averaged normal vector. (b): With averaged normal vector and co-normal vectors. (c): With averaged normal vector from previous step projected to the plane perpendicular to $\dtv[E]$.}
	\label{fig:angle_computation}
\end{figure}
For the last equality we used the identity $\sphericalangle(\Av{\nu},\nu_T)\circ\phi=\frac{\pi}{2}-\sphericalangle(\Av{\nu},\mu_{T})\circ\phi$, see Figure~\ref{fig:angle_computation} (a) and (b). Now we expect $\Av{\nu}\circ\phi\cdot\dcnv[T]\to 0$, which is numerically stable. Formulation \eqref{eq:angle_rewritten_a} still has the crucial drawback that information of both neighbored elements are required to compute the averaged normal vector. This vector, however, acts only as an auxiliary object to compute the full angle. We could use every vector $v\in\R^3$ lying in between $\dcnv[L]$ and $\dcnv[R]$, i.e., $v\in\mathrm{conv}\{\dcnv[L],\dcnv[R]\}$, cf. Figure~\ref{fig:angle_computation} (c), boundary conditions are discussed in Section~\ref{subsec:bc}. Therefore, as proposed in \cite{NS19} we use the orthogonal projector $\bm{P}_{\dtv[E]}^\perp:=\bI-\dtv[E]\otimes \dtv[E]$, and the nonlinear operator
\begin{align}
	\label{eq:projection_av_nv}
	P^\perp_{\dtv[E]}(\Av{\nu}^n):=\frac{1}{\|\bm{P}_{\dtv[E]}^\perp\Av{\nu}^n\|}\bm{P}_{\dtv[E]}^\perp\Av{\nu}^n.
\end{align}
Here $\Av{\nu}^n$ denotes the averaged normal vector from e.g. the previously computed solution. The projection is needed to guarantee that the auxiliary vector lies in the plane perpendicular to $\dtv[E]$ to obtain the correct angle. Averaging after each completed load-step has proven its worth as the elements in practice do not rotate such that $\Av{\nu}^n$ does not lie in between the deformed co-normal vectors anymore. The averaging procedure can also be performed after each Newton step making the method more robust at the cost of solving several averaging problems. As the averaging is done edge-patch wise, however, it is cheap. Note, that the nonlinear projection \eqref{eq:projection_av_nv} depends on the unknown deformation. Angle term \eqref{eq:angle_rewritten} reads with the projection operator

\begin{align}
	\sum_{E\in\E}\int_{E}\sphericalangle(\nu_L, \nu_R)\circ\phi-\sphericalangle(\rnv_L, \rnv_R)\ds=-\sum_{T\in\T}\int_{\partial T}\sphericalangle(P^\perp_{\dtv_E}(\Av{\nu}^n), \mu)\circ\phi-\sphericalangle(\Av{\rnv}, \rcnv)\ds.\label{eq:stable_angle}
\end{align}

\subsection{Physical meaning of hybridization field}
\label{sec:physical_meaning_hyb}
Taking the variation of the edge terms from \eqref{eq:HHJ_hyb_nonlinear} with respect to $\bending$ yields
\begin{align*}
	\sum_{E\in\E}\int_E(\sphericalangle(\nu_L,\nu_R)\circ\phi-\sphericalangle(\rnv_L,\rnv_R))\mean{(\delta\bending)_{\rcnv\rcnv}}+ \angle_{\rcnv}\jmp{(\delta\bending)_{\rcnv\rcnv}}\ds=0 \,\,\, \forall  \delta\bending\in\spaceB^{\mathrm{dc}}(\T).
\end{align*}
Using the test function $(\delta\bending)_{\rcnv\rcnv}|_{T_L}=1$ and $(\delta\bending)_{\rcnv\rcnv}|_{T_R}=-1$ for edge $E=T_L\cap T_R$ and zero elsewhere, we deduce that $\angle=0$, i.e. $\angle$ is a virtual Lagrange parameter. For the stable angle computation
\begin{align*}
	\sum_{T\in\T}\int_{\partial T}\big(\sphericalangle(P^\perp_{\dtv_E}(\Av{\nu}^n), \mu)\circ\phi-\sphericalangle(\Av{\rnv}, \rcnv)+\angle_{\rcnv}\big)(\delta\bending)_{\rcnv\rcnv}\ds=0\quad \forall \delta\bending\in\spaceB^{\mathrm{dc}}(\T)
\end{align*}
we obtain $\angle_{\rcnv_E}=\frac{1}{2}\jmp{\sphericalangle(P^\perp_{\dtv_E}(\Av{\nu}^n), \mu)\circ\phi}$. It measures how much $\Av{\nu}^n$ has to be rotated to obtain $\Av{\nu}\circ \phi$. When performing the averaging procedure at every Newton step $\angle$ becomes again a virtual Lagrange multiplier after convergence.

\subsection{Moore--Penrose pseudo inverse}
\label{subsec:pseudo_inverse}
The Moore--Penrose pseudo inverse of a $3\times3$ rank 2 matrix $A$ with known kernel vector $v$, like the deformation gradient $\DGS$ with the normal vector $\rnv$ as kernel, can be computed in terms of a Tikhonov regularization
\begin{align}
	\bA^\dagger = (\bA^\top \bA + v\otimes v)^{-1} \bA^\top.
\end{align}
With it, the director $\director=\dnv+\DGS^{\dagger^\top}\rshear$ can be computed.
\subsection{Boundary conditions}
\label{subsec:bc}
To solve the shell equations numerically we need to specify appropriate boundary conditions. Due to the appearance of the less common HHJ, $\HCurl$-conforming N\'ed\'elec, and co-normal continuous hybridization space, we discuss in this section the different types of boundary conditions arising in shell problems.
\begin{table}[H]
	\begin{tabular}{c|ccc|cc}
		& $u$ & $\rshear$ & $\bending$ & $\bending^{\mathrm{dc}}$ & $\angle$\\
		\hline
		clamped          & D   & D & N & - & D\\
		free             & N   & N & D & - & N\\
		simply supported & D   & N & D & - & N\\
		symmetry         & $u_{\rcnv}=0$ & N & N & - & D\\
		rigid diaphragm  & $u_{\rtv_E}=u_{\rnv}=0$ & D & D & - & N
	\end{tabular}\hfill
	\begin{tabular}{c|cccc}
		field & $u$ & $\rshear$ & $\bending$  & $\angle$\\
		\hline
		Dirichlet & $u$   & $\rshear_{\rtv}$ & $\bending_{\rcnv\rcnv}$  &$\angle_{\rcnv}$
	\end{tabular}
	\caption{Left: Requirement on spaces. D and N correspond to essential zero Dirichlet and homogeneous natural Neumann conditions, respectively. For symmetry the co-normal component of $u$ is fixed whereas for rigid diaphragm only the co-normal component of $u$ is free. Right: Dirichlet components for the fields.}
	\label{tab:boundary_conditions}
\end{table}

In Table~\ref{tab:boundary_conditions} (left) the requirements on the different spaces are listed, where the first three columns correspond to the Naghdi and (by neglecting $\rshear$) Koiter shell without hybridization and the last two columns replace the third one when hybridization is considered. We note that the essential and natural boundary conditions swap between $\bending$ and $\angle$ during hybridization. With $\angle$ a moment can be prescribed at boundaries as natural boundary condition \cite{NS19}. The different Dirichlet boundary conditions for the used fields are depicted in Table~\ref{tab:boundary_conditions} (right).

On clamped boundaries the averaged normal vector $\Av{\nu}^n$ has to coincide with the initial normal vector $\rnv$. Otherwise, the boundary behaves like a simply supported one. For symmetry boundary conditions only its co-normal component has to be fixed, i.e., $\Av{\nu}^n\cdot\rcnv=0$, rotations in the $\rnv$-$\rtv_E$-plane are allowed. These conditions have to be incorporated during the averaging procedure. Alternatively, it is possible to introduce $\Av{\nu}^n$ as additional unknown in a vector-valued skeleton space on $\E$ including the appropriate boundary conditions such that the averaging is performed automatically in every Newton iteration. 

\subsection{Stiffness matrix}
\label{subsec:stiffness_matrix}
We stated the shell methods in terms of Lagrange functionals for a compact and clear presentation. For actually solving these equations the first and second variations of the functionals are needed. Several modern finite element software allow for automatic (symbolic) differentiation, such that this tedious work is done internally. Nevertheless, we present the variations by hand for the Koiter shell model \eqref{eq:HHJ_nonlinear} together with the stable angle computation \eqref{eq:stable_angle} in Appendix~\ref{sec:first_second_variations}. If the averaging procedure of the normal vector is done after each Newton iteration the variation terms simplify significantly. The additional hybridization and shear field arising in the other methods can be computed with less effort. 

\subsection{Extension to nonlinear material laws}
\label{subsec:nonlinear_materials}
If the material law is not invertible two-field formulation \eqref{eq:HHJ_nonlinear} cannot be used. The three-field Hu--Washizu method \eqref{eq:HHJ_bend_three_field}, however, does not require the inversion as the material law is directly applied to $\curv^{\mathrm{diff}}$. We can use discontinuous piece-wise polynomials for the discretization of $\curv^{\mathrm{diff}}$ as long as the space is a subspace of the finite element space used for the bending moment tensor $\bending$. Thus, the additional field can be statically condensed out at element level not increasing the total system size of the final stiffness matrix. The idea of using a lifting from distributional to more regular quantities to apply nonlinear, not-invertible operators has already been successfully used e.g. in \cite{NPS21} for the TDNNS method for nonlinear elasticity.

\section{Structures with kinks and branched shells}
\label{sec:kinks_branched_shells}
In the previous sections we derived the models starting from a smooth shell. The question of applicability of these methods in the case of kinked or branched shells arise, where the structure is non-smooth.

We focus in this section on Koiter shells but emphasize that the results apply to Naghdi shells as well. For ease of presentation we consider the normal vectors for the angle computation \eqref{eq:angle_rewritten_a}.

\begin{figure}[h]
	\centering
	\includegraphics[width=0.2\textwidth]{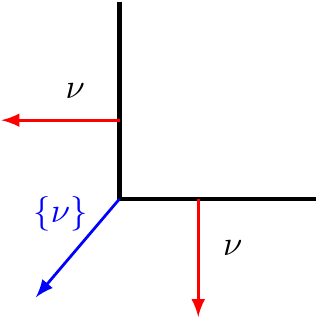}\hspace*{2cm}
	\includegraphics[width=0.2\textwidth]{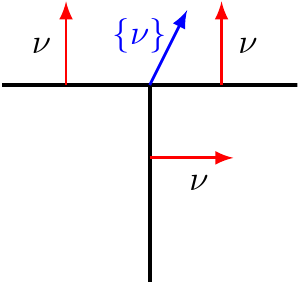}
	\caption{Left: Kinked shell. Right: Branched shell.}
	\label{fig:kink_branched_structure}
\end{figure}

\subsection{Kinked structures}
We start with kinked structures as depicted in Figure~\ref{fig:kink_branched_structure} (left). We recall that the co-normal--co-normal component of the bending stress tensor is continuous, $\jump{\bending_{\rcnv\rcnv}}=0$. Therefore, the moments entering the kink exactly go into the neighbored element. The co-normal--co-normal component ``does not see'' the junction such that the moments are preserved, which is part of the interface conditions. Additionally, if we consider the variation of $\bending$ in \eqref{eq:HHJ_nonlinear} on a single edge, we deduce from
\begin{align}
	\label{eq:kink_angle_preservation}
	\int_E(\sphericalangle(\nu_L,\nu_R)\circ\phi-\sphericalangle(\rnv_L,\rnv_R))(\delta\bending)_{\rcnv\rcnv}\,ds=0,\qquad \forall \delta\bending\in\spaceB(\T)
\end{align}
that in a weak sense the initial angle of the kink gets preserved, $\sphericalangle(\nu_L,\nu_R)\circ\phi=\sphericalangle(\rnv_L,\rnv_R)$. Thus, the proposed methods can be applied to kinked structures without the necessity of adaptions at the kinks, see Section~\ref{subsec:t_cant} for a numerical example. For Naghdi shells condition \eqref{eq:kink_angle_preservation} changes such that the angle of the directors $\director$ is preserved.

\subsection{Branched shells}
For branched shells, where edges are shared by more than two elements, the situation is more involved and the question of how to define the boundary term arises. Therefore, we first define the averaged normal vector $\Av{\nu}\circ\phi=\frac{\sum_{i=1}^N\dnv[i]}{\|\sum_{i=1}^N\dnv[i]\|}$, where $N$ denotes the number of elements connected to the edge and $\dnv[i]$ corresponding normal vectors of the branches, whose orientations are a priori fixed (e.g. during mesh generation), cf. Figure~\ref{fig:kink_branched_structure} (right). In this way also the possibility of cancellation, $\sum_{i=1}^N\dnv[i]=0$, can easily be avoided by changing the orientation of one branch. Motivated by \eqref{eq:kink_angle_preservation} a first approach would be
\begin{align}
	\label{eq:branched_shells_first}
	\sum_{i=1}^N\int_{\d T_i\cap E}(\sphericalangle(\nu_i,\Av{\nu})\circ\phi-\sphericalangle(\rnv_i,\Av{\rnv}))\bending_{\rcnv\rcnv}\ds.
\end{align}
Taking again the variation of $\bending$ would yield that the sum of the changed angles is zero, but not each angle itself. To force that each integral in the sum of \eqref{eq:branched_shells_first} is zero we make $\bending$ discontinuous in terms of hybridization, cf. Section~\ref{subsec:hybridization}. Then we obtain that each angle is (weakly) preserved. 

The Lagrange multiplier $\angle$, which reinforces the co-normal--co-normal continuity of $\bending$ for non-branched shells, now guarantees that the sum of moment inflows is equal to the sum of moment outflows at the edge
\begin{align}
	\sum_{T\in\T : E\subset T}\int_{\d T\cap E}(\delta\angle)_{\rcnv}\bending_{\rcnv\rcnv}\ds=0\qquad \forall\delta\angle\in\spaceH(\T).
\end{align}
This equilibrium of moments leads to a physically correct behavior of shells. As a result hybridization is recommended and needed for branched shells. Nevertheless, no further action is required to handle these types of structures.

We emphasize that the linearization result Lemma~\ref{lem:linearization_angle_terms} holds also for kinked and (in the hybridized case) branched shells. Thus, the linearized methods \eqref{eq:HHJ_linear_shell} and \eqref{eq:TDNNS_linear_shell} can also directly be used for these structures.

\section{Finite elements}
\label{sec:finite_elements}
In this section we state the finite element spaces needed  to approximate and solve the shell problems appropriately and discuss a procedure to alleviate membrane locking. 

\subsection{Definition of finite element spaces}
For the finite element computations we make use of the unit reference triangle $\refT$ and reference edge $\refE$ on which we define with $\Pol^k(\refT)$ and $\Pol^k(\refE)$ the set of polynomials up to degree $k$. Let $\re{T}\in\T$ denote a possibly curved triangle in $\R^3$ which is diffeomorphic to $\refT$ via $\refPhi_{\re{T}}:\refT\to\re{T}$. We say that $\re{T}$ is curved of degree $p$ if $\refPhi_{\re{T}}\in\Pol^p(\refT,\R^3)$.

The Lagrangian finite elements $\spaceU_h^k\subset\spaceU(\T)$ on $\T$ are given by, see e.g. \cite{Zaglmayr06,Bath14,ZT20},
\begin{align}
	\label{eq:fes_U}
	\spaceU^k_h := \{u\in\spaceU(\T)\,\vert \forall \re{T}\in\T,\, u\circ\refPhi_{\re{T}} =\re{u} \text{ for a } \re{u}\in\Pol^k(\refT,\R^3) \}.
\end{align}
Next, we define the Hellan--Herrmann--Johnson finite element space \cite{PS11,Com89} $\spaceB_h^k\subset\spaceB(\T)$
\begin{align}
	\label{eq:fes_B}
	\spaceB^k_h := \{\bending\in\spaceB(\T)\,\vert \forall \re{T}\in\T,\, \bending\circ\refPhi_{\re{T}} =\frac{1}{J^2}\DG\re{\bending}\DG^\top \text{ for a } \re{\bending}\in\Pol^k(\refT,\R^{2\times 2}_{\mathrm{sym}}) \},
\end{align}
where $\DG=\nabla_{\tilde{x}}\refPhi_{\re{T}}\in\R^{3\times 2}$ is the classical push forward and $J=\sqrt{\det (\DG^\top\DG)}$ the surface determinant. We emphasize that $\bending(\refPhi_{\re{T}}(\tilde{p}))=(\frac{1}{J^2}\DG\re{\bending}\DG^\top)(\tilde{p})\in T_{\refPhi_{\re{T}}(\tilde{p})}\re{T}\otimes T_{\refPhi_{\re{T}}(\tilde{p})}\re{T}$ lies in the tangent space of $\re{T}$ due to the push forward $\DG$. Further, $\bending$ remains symmetric. Moreover, with the scaling by the determinant the co-normal--co-normal continuity is preserved as $\frac{1}{J}\DG$ is well known as the (contravariant) Piola transformation used to preserve the normal continuity e.g. for $\HDiv$-conforming Raviart--Thomas or Brezzi--Douglas--Marini elements \cite{BDM85,RT77}. The Piola transformation is also used for the hybridization space $\spaceH_h^k\subset\spaceH(\T)$
\begin{align}
	\label{eq:fes_H}
	\spaceH^k_h := \{\angle\in\spaceH(\T)\,\vert \forall \re{E}\in\E,\, \angle\circ\refPhi_{\re{E}} =\frac{1}{J}\DG(\re{f}\tilde{\mu}_E) \text{ for a } \re{f}\in\Pol^k(\refE) \},
\end{align}
where $\tilde{\mu}_E$ is the a-priori fixed edge normal vector, see Section~\ref{subsec:discrete_def_surfaces}.

For the shearing space we use $\HCurl$-conforming N\'ed\'elec elements \cite{Ned1980} mapped onto the surface. Due to their tangential continuity the covariant transformation is considered
\begin{align}
	\label{eq:fes_S}
	\spaceS^k_h = \{\beta\in\spaceS(\T)\,\vert \forall \re{T}\in\T,\, \beta\circ\refPhi_{\re{T}} =\DG^{\dagger^\top}\re{\beta} \text{ for a } \re{\beta}\in\Pol^k(\refT,\R^2) \}.
\end{align}
Note, that $\DG^{\dagger^\top}=\DG(\DG^\top\DG)^{-1}\in\R^{3\times 2}$ maps into the tangent space of $\re{T}$.

\begin{figure}[h!]
	\centering
	\includegraphics[width=0.155\textwidth]{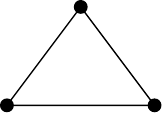}
	\includegraphics[width=0.155\textwidth]{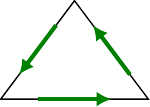}
	\includegraphics[width=0.155\textwidth]{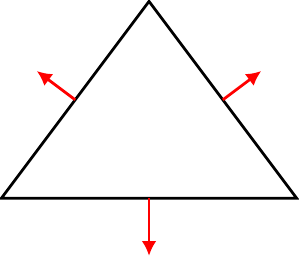}
	\includegraphics[width=0.155\textwidth]{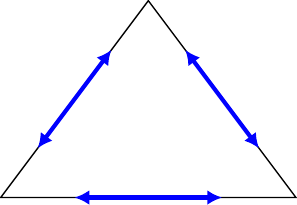}
	\includegraphics[width=0.155\textwidth]{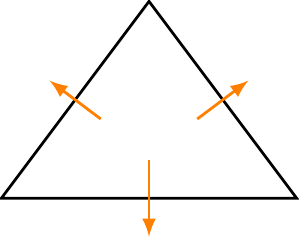}
	\includegraphics[width=0.155\textwidth]{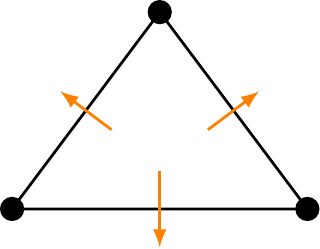}
	
	\includegraphics[width=0.155\textwidth]{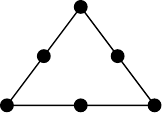}
	\includegraphics[width=0.155\textwidth]{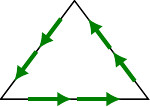}
	\includegraphics[width=0.155\textwidth]{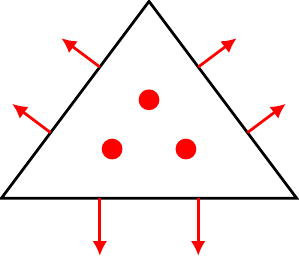}
	\includegraphics[width=0.155\textwidth]{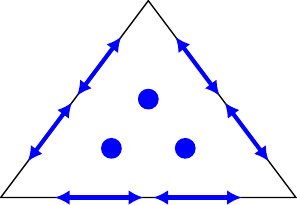}
	\includegraphics[width=0.155\textwidth]{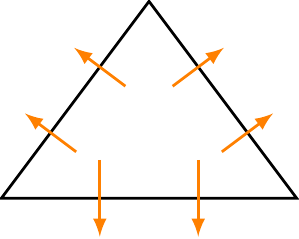}
	\includegraphics[width=0.155\textwidth]{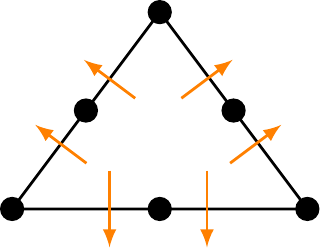}
	
	\caption{Schematic dofs of used finite elements from left to right: Lagrangian, N\'ed\'elec, HHJ, Regge, hybridization, hybridized Koiter shell element. Top: lowest order, bottom: one polynomial order higher.}
	\label{fig:finite_elements}
\end{figure}

The Regge finite element space used to alleviate membrane locking, cf. Section~\ref{sec:membrane_locking} consists, in comparison with HHJ finite elements, of tangential--tangential instead of co-normal--co-normal continuous functions \cite{Li18,Neun21}. In the two-dimensional setting this corresponds to a rotation of $90\degree$ of the shape functions. Further, to preserve the tangential--tangential continuity a double covariant mapping is used
\begin{align}
	&\Regge(\T):=\{ \bm{\veps}\in\Cinf[\T,T\rSurf\otimes T\rSurf]\,\vert\, \bm{\veps}^\top=\bm{\veps},\,\jump[E]{\bm{\veps}_{\rtv\rtv}}=0\forall E\in\E^{\mathrm{int}}\}, \label{eq:space_Regge}\\
	&\Regge_h^k:=\{ \bm{\veps}\in \Regge(\T)\,\vert\, \forall\re{T}\in\T,\,\bm{\veps}\circ\refPhi_{\re{T}}=\DG^{\dagger^\top}\re{\bm{\veps}}\DG^\dagger \text{ for a } \hat{\bm{\veps}}\in\Pol^k(\refT,\R^{2\times 2}_{\mathrm{sym}}) \}.\label{eq:fes_regge}
\end{align}

In Figure~\ref{fig:finite_elements}  dofs of the presented finite elements are shown schematically and in the last column the hybridized Koiter shell element dofs are depicted. For the hybridized Naghdi shell element the shearing N\'ed\'elec dofs have to be added.

The finite elements can be defined for quadrilateral elements in the same manner using the appropriate transformations from the reference quadrilateral. As the polynomial spaces on the reference quadrilateral differ between the used elements we refer to the literature \cite{BBF13,Zaglmayr06,PS12,Neun21} for details. We emphasize that for properly designed shape functions meshes mixing both triangular and quadrilateral elements are straight-forward.

\subsection{Membrane locking and Regge interpolation}
\label{sec:membrane_locking}
It is well-known that for shells so-called membrane and shear locking might occur if the thickness parameter $t$ becomes small and the deformation falls in the bending dominated regime \cite{CB98}. For Koiter shells no shear locking occurs as the shearing  unknowns have been eliminated by the Kirchhoff--Love hypothesis. The proposed method for the nonlinear Naghdi shells is due to the hierarchical approach, and as will be demonstrated in Section~\ref{sec:numerics} by means of several numerical examples, also free of shear locking. Both of them, however, still suffer from membrane locking if displacement functions of at least quadratic polynomial order are considered. Using high enough polynomial ansatz functions is known to mitigate this problem, see e.g. \cite{choi98}. Nevertheless, we will use the Regge interpolation operator as proposed in \cite{NS21} to obtain a locking free method independent of polynomial degree. Therefore, the piece-wise canonical interpolant $\RegInt$ into the Regge finite elements \eqref{eq:fes_regge} is inserted into the membrane energy term
\begin{align}
	\label{eq:membrane_energy_regge}
	E_{\mathrm{mem}}^{\mathrm{Regge}}(u) = \frac{t}{2}\int_{\T}\|\RegInt \GTS(u)\|^2_{\MT}\dx
\end{align}
reducing the implicitly given kernel constraints, where the polynomial degree $k$ is chosen one order less than for the displacement field $u$. In Section~\ref{sec:numerics} numerical examples verify that with this adaption the methods are free of membrane locking.
The canonical Regge interpolant $\RegInt[k]:\Regge(\T)\to \Regge_h^k$ on a $\re{T}\in\T$ is given by the following equations
\begin{align*}
	&\int_{\re{E}} (\RegInt[k]\bm{\veps})_{\rtv\rtv}\,q\ds = \int_{\re{E}}\bm{\veps}_{\rtv\rtv}\,q\ds&&\forall q\circ\refPhi_{\re{E}}=\frac{1}{J_{\mathrm{bnd}}}\re{q},\,\re{q}\in\Pol^k(\refE),\,\re{E}\in\E\cap\d\re{T},\\
	&\int_{\re{T}} (\RegInt[k]\bm{\veps}):Q\dx = \int_{\re{T}}\bm{\veps}:Q\dx&&\forall Q\circ\refPhi_{\re{T}}=\frac{1}{J}\DG\re{Q}\DG^\top,\,\re{Q} \in\Pol^{k-1}(\refT,\R^{2\times 2}_{\mathrm{sym}}),
\end{align*}
where $\refPhi_{\re{E}}:\tilde{E}\to\re{E}$ maps the reference edge to the physical one and $J_{\mathrm{bnd}}:=\|\DG\tilde{\tau}\|_2$ is the corresponding edge measure.

We emphasize that the Green strain tensor $\GTS$ is an element of $\Regge(\T)$ as the discretized displacement $u\in\spaceU_h$ is continuous and thus its gradient $\DGS$ tangential continuous. We refer to \cite{NS21,Neun21} for details on the implementation and note that this procedure is related to the tying point approach of MITC (mixed interpolation of tensorial components) shell elements \cite{CB11}. It can be applied to quadrilaterals (and also for mixed meshes) in the same manner as for triangles.

\subsection{Relation to Morley triangle}
In the lowest order case using the hybridized Koiter shell model, i.e., $u\in\spaceU_h^1$, $\bending\in\spaceB_h^{\mathrm{dc},0}$, $\angle\in\spaceH_h^0$, after static condensation and eliminating the bending moment dofs at element level, the remaining dofs are equivalent to those of the Morley triangle \cite{Mor71}, cf. Figure~\ref{fig:finite_elements} (top right). There the displacements are placed at the vertices and the normal derivative of $u$, $\frac{\d u}{\d\rcnv}$, is continuous across elements at the edge mid-points. The fact that the hybridization unknown $\angle$ was shown to have the physical meaning of the normal derivative of the displacement $\angle\hat{=}\frac{\d u}{\d \rcnv}$ in the case of plates \cite{Com89} underlines this strong relationship.

\section{Numerical examples}
\label{sec:numerics}
The presented methods are implemented in the open source finite element library NGSolve\footnote{www.ngsolve.org} \cite{Sch97,Sch14}. Elements are curved isoparametrically according to the used polynomial order for the displacements. We use e.g. the abbreviation $p1$ HHJ to indicate that the Hellan--Herrmann--Johnson method with $(u_h,\bending_h,\angle_h)\in \spaceU_h^1\times \spaceB^{\mathrm{dc},0}_h\times \spaceH_h^0$  is considered for the (nonlinear) Koiter shell. $p2$ TDNNS will denote the TDNNS method with $(u_h,\bending_h,\angle_h,\rshear_h)\in \spaceU_h^2\times \spaceB^{\mathrm{dc},1}_h\times \spaceH_h^1\times \spaceS_h^1$ for the (nonlinear) Naghdi shell and elements are curved quadratically. For all benchmarks the Regge interpolant of one polynomial order less than the displacement is used for the membrane energy to alleviate membrane locking. Due to the hierarchical approach no shear locking occurs. For all examples the shear correction factor is fixed to $\kappa=5/6$.

If nonlinear benchmarks are considered Newton's method is applied with stopping criterion $\sqrt{|(r_n,A^{-1}_nr_n)|}<1\times 10^{-5}$, where $r_n$ denotes the residuum and $A_n$ the linearization of the $n$-th iteration. Uniform load-steps in $[0,1]$ with possible damping are performed.

\subsection{Cantilever subjected to end shear force}
\label{subsec:cant_shear_force}
\begin{figure}
	\centering
	\includegraphics[width=0.55\textwidth]{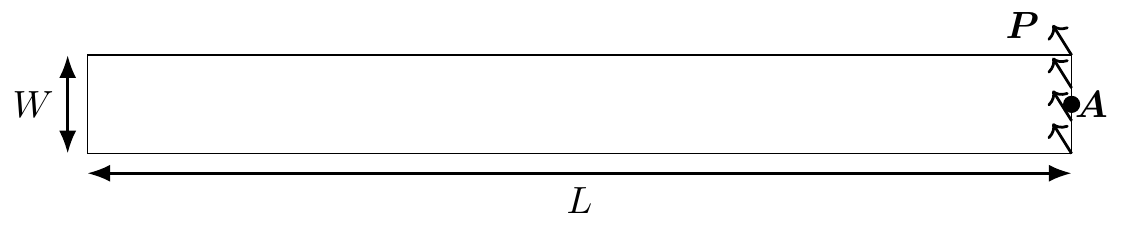}\hspace*{0.5cm}
	\includegraphics[width=0.3\textwidth]{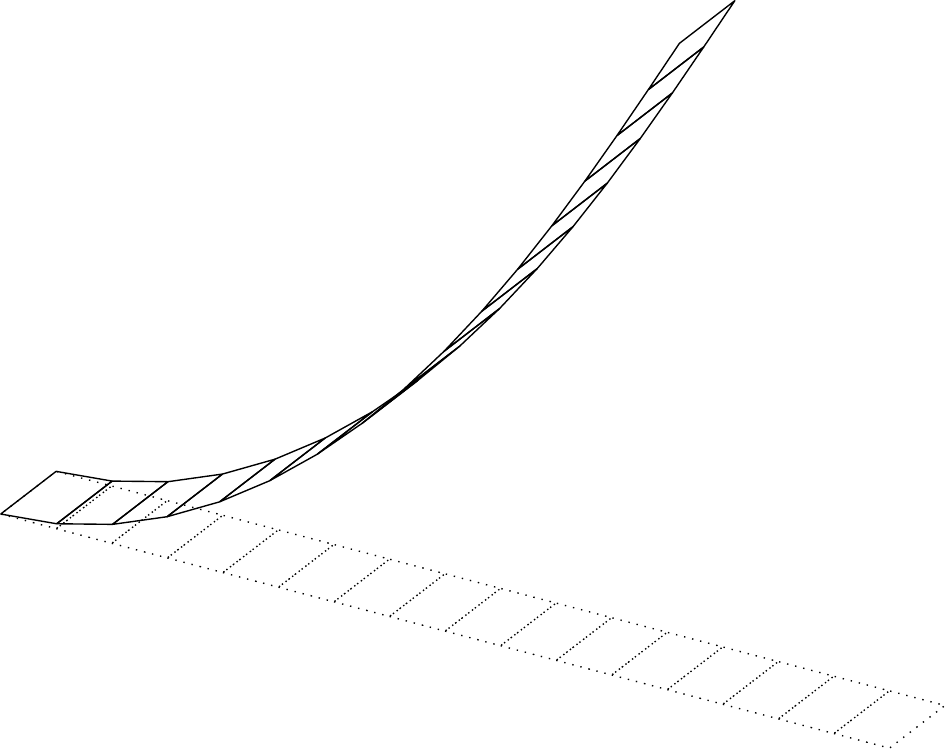}
	\caption{Left: Geometry of cantilever subjected to end shear force benchmark. Right: Initial (dotted) and deformed (solid) mesh.}
	\label{fig:res_cant_shear_mesh}
\end{figure}

An end shear force $P$ on the right boundary is applied to a cantilever, which is fixed on the left. The material and geometrical properties are $\YoungsModulus=1.2\times10^6$, $\PoissonRatio=0$, $L=10$, $W=1$, $t=0.1$, and $P_{\mathrm{max}}=4$, see Figure \ref{fig:res_cant_shear_mesh} (left). We use a structured $16\times 1$ quadrilateral grid, the lowest order $p1$ TDNNS method for nonlinear Naghdi shells, and compare it with the reference values from \cite{SLL2004} at point $A$. In Figure \ref{fig:res_cant_shear_mesh} (right) the initial and deformed mesh are displayed. The results shown in Figure \ref{fig:res_cant_shear} are in common with the reference values, the lines are overlapping. In \cite[Section 3.1]{NS19} this benchmark has been performed for the HHJ method.

\begin{figure}
	\centering
	\includegraphics[width=0.45\textwidth]{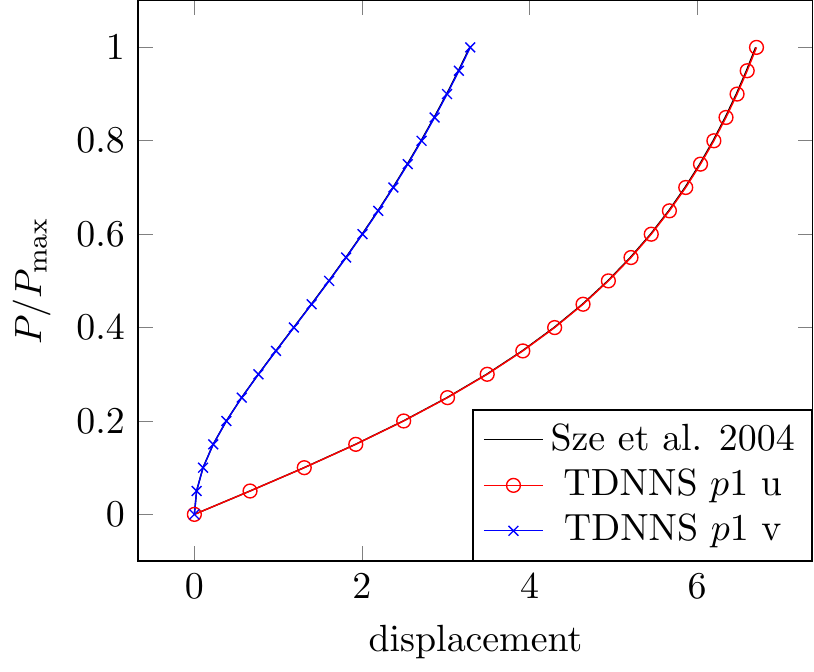}
	\caption{Horizontal and vertical load-deflection for cantilever subjected to end shear force with $16\times 1$ grid.}
	\label{fig:res_cant_shear}
\end{figure}

\subsection{Cantilever subjected to end moment}
\label{subsec:cant_end_mom}
\begin{figure}
	\centering
	\includegraphics[width=0.5\textwidth]{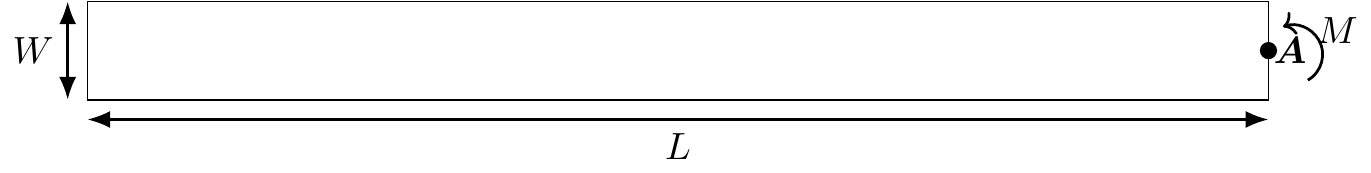}\hspace*{0.5cm}
	\includegraphics[width=0.33\textwidth]{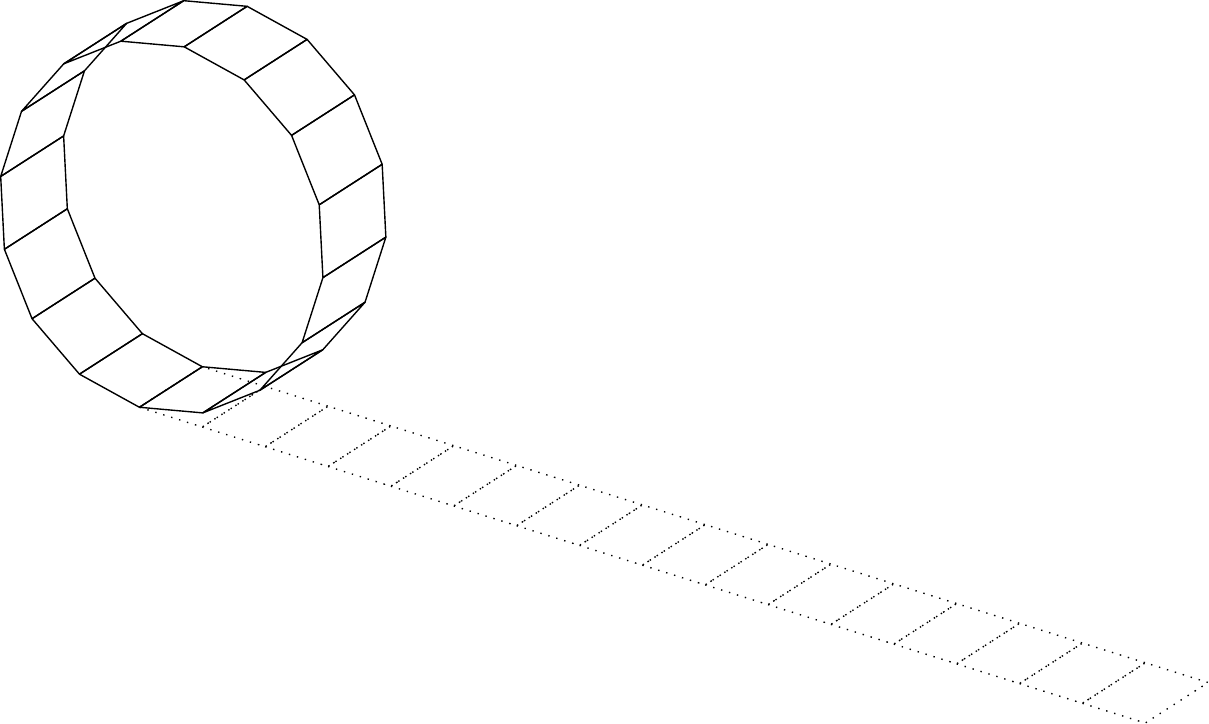}
	\caption{Left: Geometry of cantilever subjected to end moment benchmark. Right: Initial (dotted) and deformed (solid) mesh.}
	\label{fig:res_cant_bend_mesh}
\end{figure}
A cantilever is clamped on the left side and a moment $M$ is applied on the right. The material and geometrical properties are $\YoungsModulus=1.2\times10^6$, $\PoissonRatio=0$, $L=12$, $W=1$, $t=0.1$, and $M_{\mathrm{max}}=50\pi/3$, see Figure \ref{fig:res_cant_bend_mesh} (left). We use a structured $16\times 1$ quadrilateral mesh with TDNNS $p1$ and $p2$ method, the initial and final mesh can be found in Figure \ref{fig:res_cant_bend_mesh} (right). The results are displayed in Figure \ref{fig:res_cant_bend}, which are compared with the analytic solution of this benchmark, see e.g. \cite{SLL2004}. Again, the curves overlap showing that the method is free of locking and resembles pure bending. For the results for the HHJ method we refer to \cite[Section 3.2]{NS19}.

\begin{figure}
	\centering
	\includegraphics[width=0.45\textwidth]{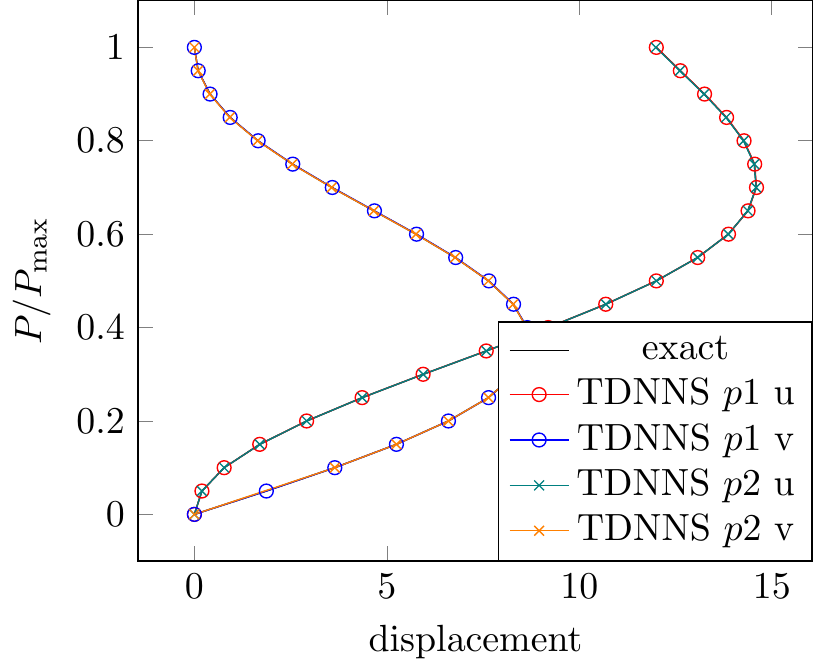}
	\caption{Horizontal and vertical load-deflection for cantilever subjected to end moment with $16\times 1$.}
	\label{fig:res_cant_bend}
\end{figure}

\subsection{Slit Annular plate}
\label{subsec:slit_annular}
We consider the slit annular plate, where one of the slit parts are clamped whereas on the other a vertical force $P$ is applied. The other two boundaries are left free. The material and geometrical properties are $\YoungsModulus=2.1\times10^8$, $\PoissonRatio=0$, $R_i=6$, $R_o=10$, $t=0.03$, and $P_{\mathrm{max}}=4.034$, see Figure \ref{fig:res_slit_ann_mesh} (left). We use unstructured triangular meshes, the initial and deformed mesh for mesh-size $h=2$ can be seen in Figure~\ref{fig:res_slit_ann_mesh} (right). Due to the large deformation we damp the first two Newton iterations with a factor $1/3$ and $2/3$, respectively. As shown in Figure~\ref{fig:res_slit_annular_plate} already the lowest order HHJ and TDNNS methods $p1$ on a coarse grid $h=2$ consisting of $98$ triangles lead to a good qualitative behavior. Using finer grids as $h=0.25$, corresponding to $6992$ triangles, resembles the reference values depicted from \cite{SLL2004}. Due to the small thickness the relative difference between the nonlinear HHJ and TDNNS method is in the range of $0.5\times 10^{-5}$.
\begin{figure}[h!]
	\centering
	\includegraphics[width=0.49\textwidth]{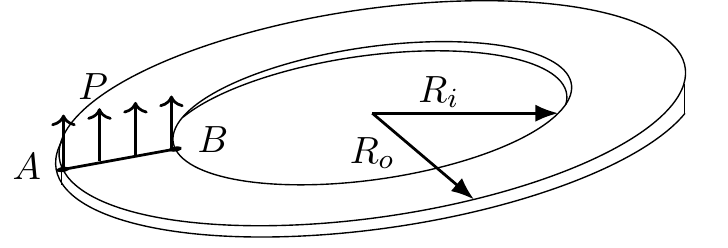}\hspace*{0.5cm}
	\includegraphics[width=0.35\textwidth]{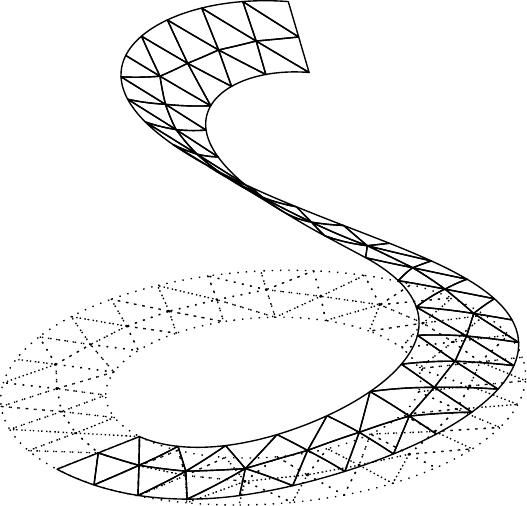}
	\caption{Left: Geometry, force, and points of interest of slit annular plate. Right: Initial (dotted) and deformed (solid) mesh.}
	\label{fig:res_slit_ann_mesh}
\end{figure}

\begin{figure}
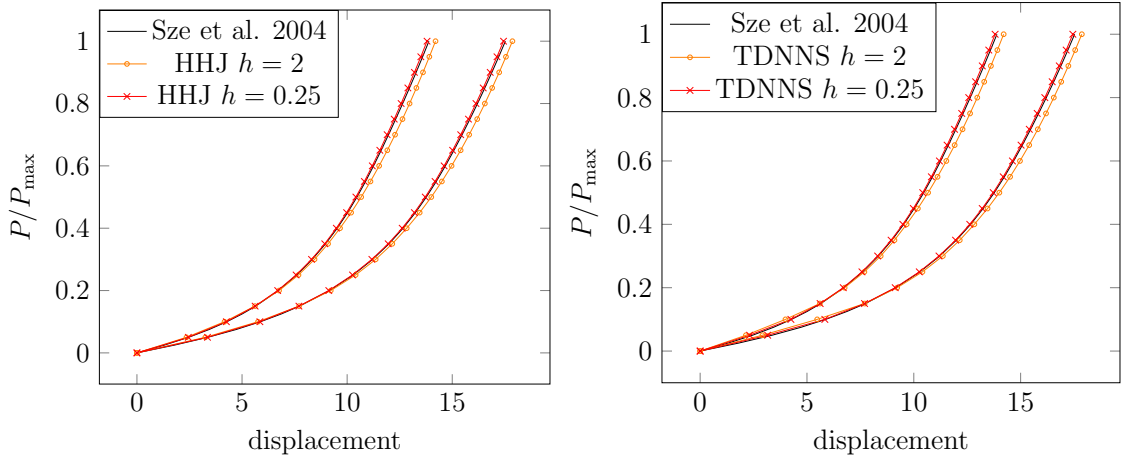

	\include{results_slit_annular}
	\caption{Vertical deflection of points $A$ and $B$ of slit annular plate for $p1$ HHJ (left) and $p1$ TDNNS (right).}
	\label{fig:res_slit_annular_plate}
\end{figure}

\subsection{Pinched cylinder}
\label{subsec:pinched_cyl}
The pinched cylinder is a well-known benchmark for geometrically nonlinear shell elements, see e.g. \cite{OSRB2017, SLL2004, JEON2015}. The left side is clamped, whereas the right is free and a point force acts on the top and bottom pointing inwards. Due to symmetry only one half of the cylinder is considered and symmetry boundary conditions are prescribed on the resulting two boundaries, cf. Figure~\ref{fig:pinched_cyl_mesh} (left). The parameters are $R=103.1$, $L=304.8$, $t=3$, $\YoungsModulus=2068.5$, $\PoissonRatio=0.3$, and $P_z=2000$ (same, scaled results are obtained for $R=1.031$, $L=3.048$, $t=0.03$, and $\YoungsModulus=2.0685\times 10^7$). Reference values are taken from \cite{SLL2004}. Due to the clamped boundary and the parabolic geometry the benchmark is in the membrane dominated regime such that no shear or membrane locking is expected to occur. Nevertheless, the Regge interpolant is used in the membrane energy to show that no spurious zero energy modes are induced. For this example we consider structured $16\times 16$ and $32\times32$ quadrilateral meshes for the methods $p1$ and $p2$ HHJ and TDNNS, respectively. The initial and deformed configuration is illustrated in Figure~\ref{fig:pinched_cyl_mesh} (right). In Figure~\ref{fig:res_pinched_cylinder} the radial deflection of point $A$ is depicted. For the first 12 of 30 loadsteps a dampfactor of $\min\{0.1\times \mathrm{it},1\}$ is considered and for the later steps it is relaxed to $\min\{0.25\times \mathrm{it},1\}$. The lowest-order method $p1$ together with the coarse $16\times 16$ has problems resembling the reference solution for the HHJ and TDNNS, however, still reflects the qualitative behavior. Increasing the polynomial order or using a finer grid resolves this issue and the results coincide well with the reference values.
\begin{figure}
	\centering
	\includegraphics[width=0.4\textwidth]{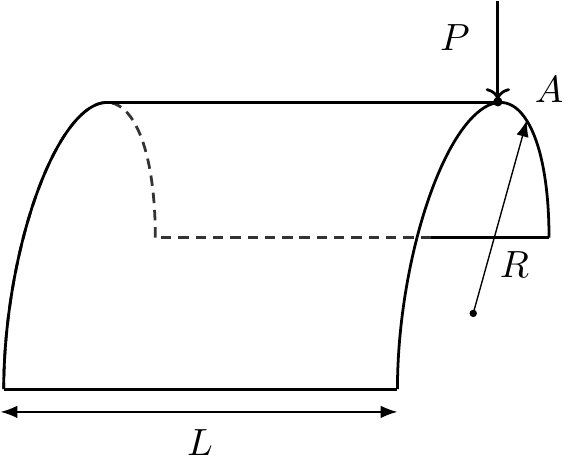}\hspace*{0.5cm}
	\includegraphics[width=0.4\textwidth]{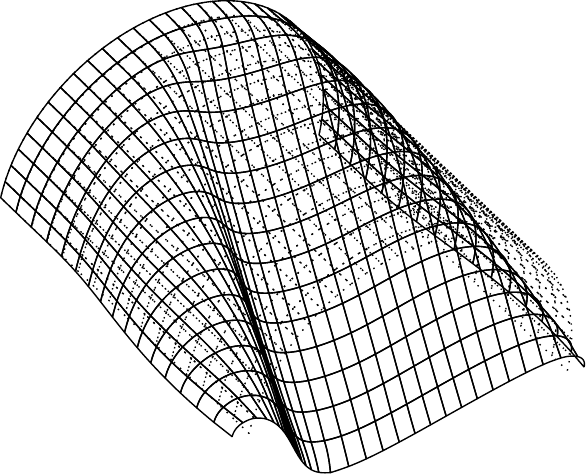}
	\caption{Left: Geometry of pinched cylinder. Right: Initial (dotted) and deformed (solid) mesh.}
	\label{fig:pinched_cyl_mesh}
\end{figure}

\begin{figure}
	\include{results_pinched_cyl}
	\caption{Radial deflection at point $A$ of pinched cylinder. Left: HHJ  method. Right: TDNNS method.}
	\label{fig:res_pinched_cylinder}
\end{figure}

\subsection{T-cantilever under shear force}
\label{subsec:t_cant}
We consider a T-cantilever as test case for a non-smooth structure with kinks, where three branches are connected by a single edge, cf. Figure~\ref{fig:tcant_mesh} (left). The parameters read $L=1$, $\PoissonRatio=0$, $\YoungsModulus=6\times 10^6$, $t=0.1$, and a vertical force $P_{\mathrm{max}}=3\times 10^3\mat{1 & 0 & 1}^\top$ is applied on the top left edge resulting to a strong rotation of the structure, see Figure~\ref{fig:tcant_mesh} (right). The bottom boundary is clamped, the others are left free. The $x$-displacement of point $A$ and $z$-displacement of $B$ are depicted in Figure~\ref{fig:tcant_results} (top) for unstructured triangular meshes for $p3$ ($h=0.5$ and $h=0.25$ corresponding to 14 and 62 elements) and are displayed in Table~\ref{tab:refvalues_tstruct} (with $h=0.125$ corresponding to 278 elements). The modulus of the bending moment tensor $\bending$ is shown in Figure~\ref{fig:tcant_results} (bottom). Due to the large deformation, we damp the first two Newton iterations with a factor 1/3 and 2/3, respectively. One can clearly observe that the moments induced by the shear force get all absorbed by the clamped boundary at the bottom, the right branch undergoes a pure rotation such that the moment tensor is zero there. Further, the initial $90^{\degree}$ angle at the kink, where the branches meet, gets preserved. For this example the difference between the nonlinear HHJ and TDNNS methods are marginally.
\begin{figure}
	\centering
	\includegraphics[width=0.4\textwidth]{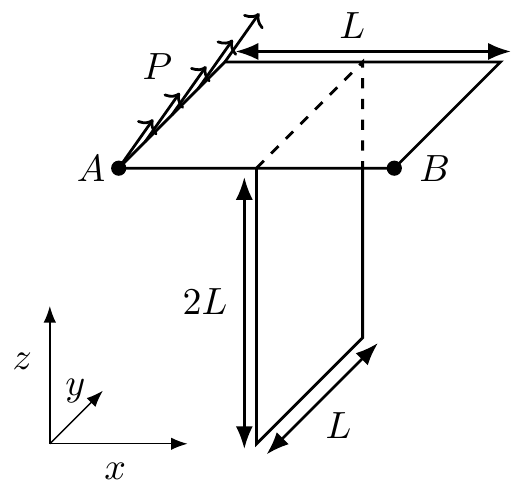}\hspace*{0.5cm}
	\includegraphics[width=0.4\textwidth]{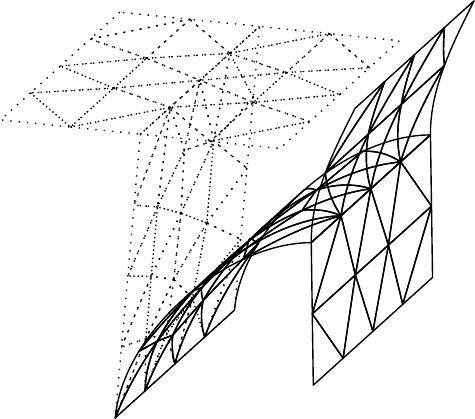}
	\caption{Left: Geometry of T-cantilever. Right: Initial (dotted) and deformed (solid) mesh.}
	\label{fig:tcant_mesh}
\end{figure}

\begin{figure}
	\centering
	\include{results_t_struct}
	\caption{Top: $x$ displacement of $A$ and $z$-displacement of $B$ for $p3$ HHJ and TDNNS. Bottom: Modulus of bending moment tensor at final configuration.}
	\label{fig:tcant_results}
\end{figure}

\begin{table}
	\scalebox{0.82}{\mbox{\ensuremath{\displaystyle
				\begin{tabular}{c|cc|cc||c|cc|cc}
					& $u^{\mathrm{HHJ}}_x(A)$ & $u^{\mathrm{HHJ}}_z(B)$&  $u^{\mathrm{TDNNS}}_x(A)$& $u^{\mathrm{TDNNS}}_z(B)$&& $u^{\mathrm{HHJ}}_x(A)$ & $u^{\mathrm{HHJ}}_z(B)$ & $u^{\mathrm{TDNNS}}_x(A)$ &$u^{\mathrm{TDNNS}}_z(B)$\\
					\hline
					0.05 & 0.198 & 0.168 & 0.199 & 0.168 & 0.55 & 1.113 & 0.767 & 1.129 & 0.773 \\
					0.1  & 0.409 & 0.334 & 0.410 & 0.334 & 0.6  & 1.137 & 0.777 & 1.155 & 0.784 \\
					0.15 & 0.585 & 0.461 & 0.588 & 0.462 & 0.65 & 1.158 & 0.786 & 1.178 & 0.793 \\
					0.2  & 0.721 & 0.551 & 0.725 & 0.553 & 0.7  & 1.177 & 0.792 & 1.199 & 0.801 \\
					0.25 & 0.823 & 0.615 & 0.829 & 0.617 & 0.75 & 1.193 & 0.798 & 1.218 & 0.808 \\
					0.3  & 0.901 & 0.660 & 0.908 & 0.663 & 0.8  & 1.208 & 0.803 & 1.234 & 0.813 \\
					0.35 & 0.962 & 0.694 & 0.971 & 0.697 & 0.85 & 1.221 & 0.807 & 1.250 & 0.818 \\
					0.4  & 1.011 & 0.720 & 1.022 & 0.723 & 0.9  & 1.233 & 0.810 & 1.264 & 0.823 \\
					0.45 & 1.052 & 0.740 & 1.064 & 0.744 & 0.95 & 1.243 & 0.813 & 1.277 & 0.827 \\
					0.5  & 1.085 & 0.755 & 1.099 & 0.760 & 1.0  & 1.253 & 0.815 & 1.290 & 0.830 \\
	\end{tabular}}}}
	\caption{Horizontal deflection at $A$ and vertical deflection of $B$ for different load steps for nonlinear HHJ and TDNNS method $p3$.}
	\label{tab:refvalues_tstruct}
\end{table}

\subsection{Axisymmetric hyperboloid with free ends}
\label{subsec:hyperboloid}
An axisymmetric hyperboloid described by the equation 
\begin{align}
	x^2+y^2=R^2+z^2,\quad z\in [-R,R]
\end{align}
with free boundaries is loaded by a periodic force, see e.g. \cite{CB11}. Due to symmetries it is sufficient to use one eighth of the geometry and symmetry boundary conditions are prescribed, see Figure \ref{fig:geom_hyp_shell} for the geometry and meshes. The material and geometric parameters are $R=1$, $\YoungsModulus=2.85\times 10^4$, $\PoissonRatio=0.3$, $t\in\{1,0.1,0.01,0.001\}$, and the periodic force $p=t^310^4\cos(2\zeta)\rnv$ is applied, where $\zeta\in [0,2\pi)$ denotes the angle. This benchmark is known to induce strong locking because of the hyperbolic geometry. Further, due to the free boundary condition boundary layers of magnitude $\mathcal{O}(t)$ occur for Reissner--Mindlin shells avoiding clear convergence rates. To resolve the boundary layers we place three quadrilateral layers in $0.75t$ at the free boundary and use structured triangular meshes to demonstrate the straight forward combination of triangular and quadrilateral elements, cf. Figure~\ref{fig:geom_hyp_shell} (right). In total $4\times 4$, $7\times 7$, $12\times 12$, $24\times 24$, $48\times 48$, $96\times 96$, and $192\times 192$ grids are considered. Only for $t=1$ no boundary layer mesh is used as the layers are directly resolved. Possible boundary layers of magnitude $\mathcal{O}(t^{\frac{1}{2}})$ induced by the curved geometry are negligible as from $(0.001)^{\frac{1}{2}}\approx0.03$  they get resolved by intermediate fine grids also for the smallest thickness. For comparison, structured triangular grids without boundary layer adaptions are considered too. The reference values for the linear Reissner--Mindlin shell given by the maximal radial deflection in the x-y-plane ($1.3577317$, $0.18954566$, $0.15046617$, $0.1498902$, respectively) were computed by reducing the equations to a one-dimensional problem exploiting the periodic force and solved by 1D high-order FEM (see e.g. \cite{SH2017} for an approach to derive the 1D equations).

\begin{figure}
	\centering
	\includegraphics[width=0.4\textwidth]{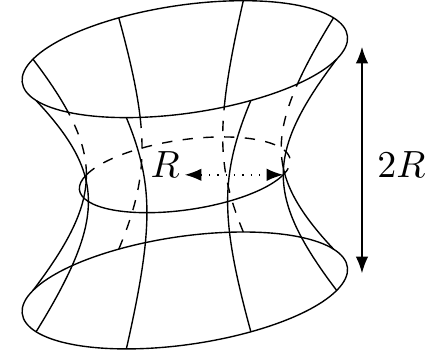}\hspace*{0.5cm}
	\includegraphics[width=0.4\textwidth]{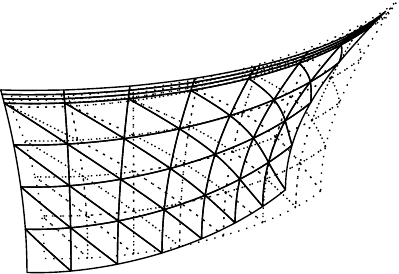}
	\caption{Left: Geometry of axisymmetric hyperboloid with free ends. Right: initial and deformed mesh for $t=0.1$ with boundary layer mesh.}
	\label{fig:geom_hyp_shell}
\end{figure}

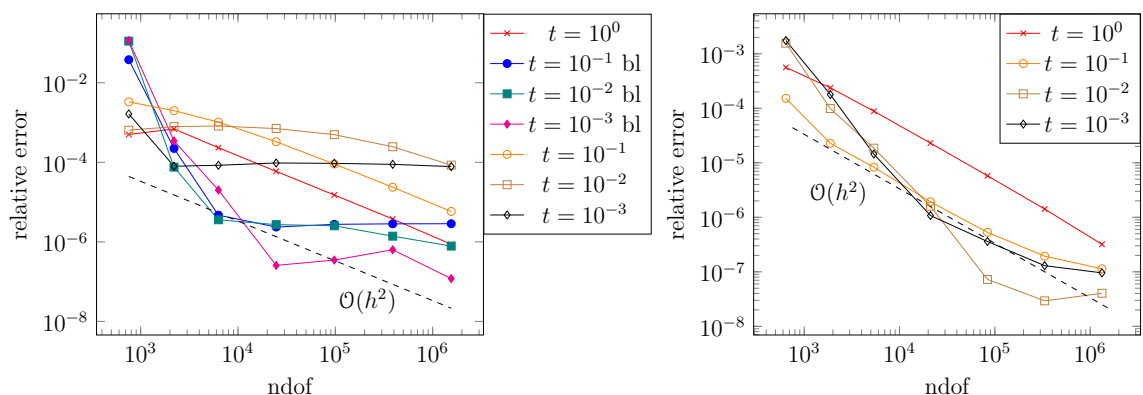
\begin{figure}
	\resizebox{0.99\textwidth}{!}{
		\begin{tikzpicture}
			\begin{axis}[
				legend style={at={(1,1)}, anchor=north west},
				xlabel={ndof},
				ylabel={relative error},
				xmode=log, ymode=log,
				]

				\addlegendentry{$t=10^0$}
				\addplot[color=red, mark=x] coordinates {
					( 755 , 0.0005001056809901336 )
					( 2201 , 0.0006908597653221863 )
					( 6291 , 0.0002328562932444221 )
					( 24675 , 5.926755246369034e-05 )
					( 97731 , 1.5057150315373797e-05 )
					( 388995 , 3.7538609331763965e-06 )
					( 1552131 , 8.689245173830874e-07 )
				};

				\addlegendentry{$t=10^{-1}$ bl}
				\addplot[color=blue, mark=*] coordinates {
					( 755 , 0.037954942584484154 )
					( 2201 , 0.00022355303830001188 )
					( 6291 , 4.6942751205848945e-06 )
					( 24675 , 2.3614399337004106e-06 )
					( 97731 , 2.7625562545571397e-06 )
					( 388995 , 2.840375013694958e-06 )
					( 1552131 , 2.858402006261053e-06 )
				};
				
				\addlegendentry{$t=10^{-2}$ bl}
				\addplot[color=teal, mark=square*] coordinates {
					( 755 , 0.11147934726153848 )
					( 2201 , 7.57557332575052e-05 )
					( 6291 , 3.589606834436527e-06 )
					( 24675 , 2.714246759994653e-06 )
					( 97731 , 2.5354951316609774e-06 )
					( 388995 , 1.3925282281808434e-06 )
					( 1552131 , 7.834255155792648e-07 )
				};
				
				\addlegendentry{$t=10^{-3}$ bl}
				\addplot[color=magenta, mark=diamond*] coordinates {
					( 755 , 0.11512660773563174 )
					( 2201 , 0.0003450574067036913 )
					( 6291 , 2.0163072643819853e-05 )
					( 24675 , 2.5645573963308006e-07 )
					( 97731 , 3.491540614350781e-07 )
					( 388995 , 6.331118444952694e-07 )
					( 1552131 , 1.1936994947570883e-07 )
				};

				\addlegendentry{$t=10^{-1}$}
				\addplot[color=orange, mark=o] coordinates {
					( 755 , 0.0033018241841378648 )
					( 2201 , 0.0019878168848974096 )
					( 6291 , 0.0010243500670593017 )
					( 24675 , 0.0003289857371609797 )
					( 97731 , 9.074188660994939e-05 )
					( 388995 , 2.361271398445586e-05 )
					( 1552131 , 5.815754202755355e-06 )
				};
				
				\addlegendentry{$t=10^{-2}$}
				\addplot[color=brown, mark=square] coordinates {
					( 755 , 0.0006411994238452574 )
					( 2201 , 0.0007916307113454231 )
					( 6291 , 0.0008166261891880194 )
					( 24675 , 0.0007065332204542134 )
					( 97731 , 0.000495693056488372 )
					( 388995 , 0.000246482268095958 )
					( 1552131 , 8.503488555007266e-05 )
				};
				
				\addlegendentry{$t=10^{-3}$}
				\addplot[color=black, mark=diamond] coordinates {
					( 755 , 0.00164926376836972 )
					( 2201 , 7.972625773660215e-05 )
					( 6291 , 8.42785320748116e-05 )
					( 24675 , 9.626725125192678e-05 )
					( 97731 , 9.41827118890767e-05 )
					( 388995 , 8.88505739401889e-05 )
					( 1552131 , 7.811557911792437e-05 )
				};
				
				\addplot[dashed,color=black, mark=none] coordinates {
					( 755 , 1/755/30 )
					( 2201 , 1/2201/30 )
					( 6291 , 1/6291/30 )
					( 24675 , 1/24675/30 )
					( 97731 , 1/97731/30 )
					( 388995 , 1/388995/30 )
					( 1552131 , 1/1552131/30 )
				};
			\end{axis}
			\node (A) at (4.8, 0.6) [] {$\mathcal{O}(h^2)$};
		\end{tikzpicture}
		\begin{tikzpicture}
			\begin{axis}[
				legend style={at={(1,1)}, anchor=north east},
				xlabel={ndof},
				ylabel={relative error},
				xmode=log, ymode=log,
				]
				
				\addlegendentry{$t=10^{0}$}
				\addplot[color=red, mark=x] coordinates {
					( 643 , 0.0005620815520611569 )
					( 1879 , 0.0002380742097911988 )
					( 5379 , 8.848961017907395e-05 )
					( 21123 , 2.29784728619802e-05 )
					( 83715 , 5.780324778764908e-06 )
					( 333315 , 1.4168941810968753e-06 )
					( 1330179 , 3.1914492273110775e-07 )
					
				};

				\addlegendentry{$t=10^{-1}$}
				\addplot[color=orange, mark=o] coordinates {
					( 643 , 0.00015252657066051062 )
					( 1879 , 2.2822434472351636e-05 )
					( 5379 , 8.31037643118966e-06 )
					( 21123 , 1.9241543449752374e-06 )
					( 83715 , 5.301776300223054e-07 )
					( 333315 , 1.935898061469691e-07 )
					( 1330179 , 1.1375496945443807e-07 )
				};
				
				\addlegendentry{$t=10^{-2}$}
				\addplot[color=brown, mark=square] coordinates {
					( 643 , 0.001565826368611172 )
					( 1879 , 9.894350689195958e-05 )
					( 5379 , 1.855560894438305e-05 )
					( 21123 , 1.59696663578528e-06 )
					( 83715 , 7.242110242490694e-08 )
					( 333315 , 2.9325637641015298e-08 )
					( 1330179 , 4.035919111479011e-08 )
				};
				
				\addlegendentry{$t=10^{-3}$}
				\addplot[color=black, mark=diamond] coordinates {
					( 643 , 0.0017490754445979306 )
					( 1879 , 0.00017906980534602464 )
					( 5379 , 1.4473516904363325e-05 )
					( 21123 , 1.0786319931350832e-06 )
					( 83715 , 3.5983366882713025e-07 )
					( 333315 , 1.2963250953641863e-07 )
					( 1330179 , 9.566918542413968e-08 )
					
				};
				
				\addplot[dashed,color=black, mark=none] coordinates {
					( 755 , 1/755/30 )
					( 2201 , 1/2201/30 )
					( 6291 , 1/6291/30 )
					( 24675 , 1/24675/30 )
					( 97731 , 1/97731/30 )
					( 388995 , 1/388995/30 )
					( 1552131 , 1/1552131/30 )
				};
			\end{axis}
			\node (A) at (1.5, 2.5) [] {$\mathcal{O}(h^2)$};
		\end{tikzpicture}
	}
	\caption{Left: Results hyperboloid for $p2$ TDNNS with (bl) and without boundary layer meshes for $t=1,0.1,0.01,0.001$. Right: $p2$ HHJ without boundary layer meshes.}
	\label{fig:result_hyperboloid}
\end{figure}

As depicted in Figure~\ref{fig:result_hyperboloid} (left) for $t=1$ quadratic convergence is obtained. Using structured triangular grids we observe a strong pre-asymptotic regime stemming from the non-resolved boundary layers, especially for $t=0.001$ the solution seems to stop converging after one refinement step as the dominant error corresponds to the boundary layer. This behavior must not be confused with locking, where the solution tends to be zero such that the pre-asymptotic regime would be close to the relative error of 1. Using quadrilaterals for resolving the boundary layers we observe a nearly uniform, fast convergence up to a relative error $<10^{-5}$ with less than 10000 dofs. As we kept the quadrilateral layer width fixed during refinement, the convergence stops when the boundary layer error becomes dominant. Using a more sophisticated adaptive strategy this behavior can be overcome, which is, however, out of scope for this work. In Figure~\ref{fig:result_hyperboloid} (right) the linear Kirchhoff--Love shell is considered with $p2$ HHJ and reference values ($0.8549465$,  $0.1856305$, $0.1502913$, and $0.1498749$). Therein a convergence to the reference values is observed without the necessity of using boundary layer meshes, as the boundary layers of thickness $\mathcal{O}(t)$ do not arise for this shell model.

\subsection{Raasch's Hook}
\label{subsec:raasch_hook}
We consider Raasch's Hook \cite{Kni1997} as a benchmark to compare the HHJ and TDNNS method for different thicknesses. In this example the linear Koiter and Naghdi shell models are used. The geometry is depicted in Figure~\ref{fig:raasch_hooke_mesh} (left), where $R_1=14$, $R_2=46$, $w=20$, and $\alpha=30^{\degree}$. We consider thicknesses $t\in\{20,2,0.2,0.02\}$ and the shear force $P$ at the right boundary is scaled by $(t/2)^3w^{-1}$, as we are in the bending dominated regime. The main deformation corresponds to a shearing, cf. Figure~\ref{fig:raasch_hooke_mesh} (right). The left boundary is clamped, the other two are free. We measure the shearing deflection of point $A$ placed at the middle of the right boundary. The values are listed in Table~\ref{tab:res_raasch_hooke}. 

We clearly observe that the methods differ for large thickness ($16.9632$ vs. $10.247$) for the smallest thickness the relative error between the models becomes less than $2\times 10^{-4}$. To minimize the discretization error the $p4$ HHJ and TDNNS method have been considered on a fine unstructured triangular grid. For the configuration $t=2$ the reference value of 5.027 for the linear Naghdi shell model proposed in \cite{KCL1998} is in common to our results.
\begin{figure}
	\centering
	\includegraphics[width=0.49\textwidth]{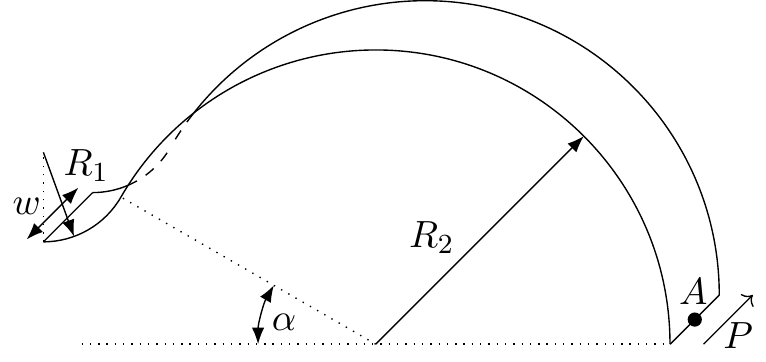}
	\includegraphics[width=0.49\textwidth]{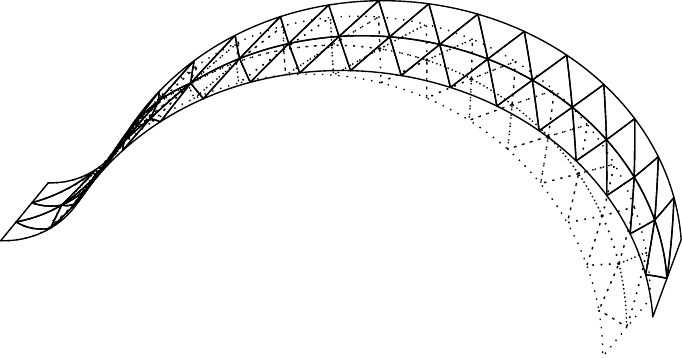}
	\caption{Left: Geometry of Raasch's Hooke. Right: Initial (dotted) and deformed (solid) mesh for $t=20$.}
	\label{fig:raasch_hooke_mesh}
\end{figure}

\begin{table}
	\centering
	\begin{tabular}{c|cccc}
		method$\backslash t$ & 20 & 2 & 0.2 & 0.02 \\
		\hline
		TDNNS & 16.9632 & 5.0382 & 4.6877 & 4.6573\\
		HHJ   & 10.2457 & 4.7188 & 4.6594 & 4.6565\\
	\end{tabular}
	\caption{$p4$ with unstructured triangular mesh $h=0.5$ consisting of 25722 triangles.}
	\label{tab:res_raasch_hooke}
\end{table}

\section{Conclusion}
In this work we presented a novel derivation of the Hellan--Herrmann--Johnson method by means of a three-field Hu--Washizu formulation, combining concepts of discrete differential geometry and (distributional) finite elements enabling nonlinear material laws. Using additional shearing degrees of freedom in a hierarchical construction leads to a shear-locking free nonlinear Naghdi shell element, which extends the linear TDNNS method for Reissner--Mindlin plates. By careful linearization of the boundary jump terms the Hellan--Herrmann--Johnson and TDNNS method for Kirchhoff--Love and Reissner--Mindlin shells, respectively, is obtained. All the proposed methods remain valid in case of a kinked or branched shell structure. Further, low-order linear displacement fields can be considered as well as arbitrary high-order elements. The drawback of mixed finite element methods leading to indefinite stiffness matrices can be overcome by hybridization techniques. Using the recently proposed interpolation of the membrane strains into the Regge finite element space avoids membrane locking. We presented several benchmark examples to demonstrate the excellent performance of the proposed shell elements.

\section*{Acknowledgments}
The authors acknowledge support by the Austrian Science Fund (FWF) project \href{https://doi.org/10.55776/F65}{10.55776/F65}. The authors thank Michael Leum\"uller for helping to spot the wrong factor in the inverted material law \eqref{eq:inverted_material_law}.

\bibliographystyle{acm}
\bibliography{cites}

\appendix
\section{Tensor cross product}
We summarize properties of the tensor cross product used in this work:
\begin{lemma}
	\label{lem:tcp_id}
	There holds for $\bA\in\R^{3\times 3}$, and $\Proj=\bI-\rnv\otimes\rnv$
	\begin{subequations}
		\begin{align}
			&\bA\tcross\bI=\tr{\bA}\bI - \bA^\top,\label{eq:tcp_id3}\\
			&\Proj\tcross\bA\rnv=\bI\tcross\bA \rnv,\label{eq:tcp_id5}\\
			&\frac{1}{2}\Proj\tcross\Proj\rnv=\rnv.\label{eq:tcp_id6}
		\end{align}
	\end{subequations}
\end{lemma}
\begin{proof}
	Direct computation, see \cite[p. 51]{BGO16}.
\end{proof}
\section{First and second variations}
\label{sec:first_second_variations}
To compute the first and second variations of the nonlinear Koiter shell model \eqref{eq:HHJ_nonlinear} with the stable angle computation \eqref{eq:stable_angle}, the different terms are explicitly derived. We start with the variations of the normalized tangent vector
\begin{align*}
	&D_u(\dtv)[\delta u]=D_u\Big(\frac{\DGS\rtv}{\|\DGS\rtv\|}\Big)[\delta u] = \frac{\rgrad\delta u\rtv}{\|\DGS\rtv\|}-\frac{\DGS\rtv}{\|\DGS\rtv\|^3} \DGS\rtv\cdot\rgrad\delta u\rtv = \bm{P}^\perp_{\tau_E}\frac{\rgrad\delta u\rtv}{\|\DGS\rtv\|},\\
	&D^2_u(\dtv)[\delta u,\Delta u]=\frac{1}{\|\DGS\rtv\|^2}\big(2\sym{\bm{P}^\perp_{\tau_E}\rgrad\Delta u\rtv\otimes \dtv}\rgrad\delta u\rtv- (\dtv\cdot\rgrad\Delta u\rtv)\bm{P}^\perp_{\tau_E}\rgrad\delta u\rtv\big).
\end{align*}
For the outer normal vector we have analogously, with $\bm{P}^\perp_{\nu} = \bI - \dnv\otimes\dnv$ and noting that with the tensor cross product there holds $D_u(\cof{\DGS}\rnv)[\delta u] = (\DGS\tcross\rgrad\delta u)\rnv$, $D^2_u(\cof{\DGS}\rnv)[\delta u,\Delta u] = (\rgrad\Delta u\tcross\rgrad\delta u)\rnv$,
\begin{align*}
	D_u(\dnv)[\delta u]= \bm{P}^\perp_{\nu}\frac{\DGS\tcross \rgrad\delta u\rnv}{\|\cof{\DGS}\rnv\|}
\end{align*}
and the second variation follows the same lines as for the tangent vector. The variations of the element normal vector follows now immediately as $\dcnv=\dnv\times\dtv$
\begin{align*}
	D_u(\dcnv)[\delta u]&=D_u(\dnv\times \dtv)[\delta u]=\bm{P}_{\nu}^\perp\frac{\DGS\tcross \rgrad\delta u\rnv}{\|\cof{\DGS}\rnv\|}\times\dtv + \dnv\times \bm{P}^\perp_{\tau_E}\frac{\rgrad\delta u\rtv}{\|\DGS\rtv\|}.
\end{align*}
The second variation is a simple but lengthy expression and thus omitted.

\subsection{First variations:}
For a more compact notation we neglect the $\phi$ dependency and write e.g., $\tau$ instead of $\dtv$. The first variations are
\begin{align*}
	&D_u\Big(\frac{t}{2}\|\GTS\|^2_{\MT}\Big)[\delta u] = 2t \MT \GTS:(\DGS^\top\rgrad\delta u),\\
	&D_{\bending}\Big(\frac{6}{t^3}\|\bending\|^2_{\MT^{-1}}\Big)[\delta\bending] = \frac{12}{t^3} \MT^{-1}\bending:\delta\bending,\\
	&D_u((\Hessian_\nu+(1-\rnv\cdot\nu)\rgrad\rnv):\bending)[\delta u]= (\rgrad^2 \delta u_i \nu_i+\Hessian_{D_u\nu[\delta u]}-\rnv\cdot D_u\nu[\delta u]\rgrad\rnv):\bending,\\
	&D_{\bending}((\Hessian_\nu+(1-\rnv\cdot\nu)\rgrad\rnv):\bending)(\delta \bending)= (\Hessian_\nu+(1-\rnv\cdot\nu)\rgrad\rnv):\delta\bending,\\
	&D_u(\sphericalangle(P_{\tau_E}^\perp(\Av{\nu}^n),\mu))[\delta u] =\\
	&\quad-\frac{1}{\sqrt{1-(P_{\tau_E}^\perp(\Av{\nu}^n)\cdot\mu)^2}}\Big(\frac{\Av{\nu}^n\cdot D_u\mu[\delta u]}{\|\Av{\nu}^n-\Av{\nu}^n\cdot\tau\tau\|} +\frac{(\Av{\nu}^n\cdot\mu)(\Av{\nu}^n\cdot D_u\tau[\delta u])(\Av{\nu}^n\cdot\tau)}{\|\Av{\nu}^n-\Av{\nu}^n\cdot\tau\tau\|^3}\Big).
\end{align*}

\subsection{Second variations:}
The second variations read
\begin{align*}
	&D^2_u\Big(\frac{t}{2}\|\GTS\|^2_{\MT}\Big)[\delta u,\Delta u] = 4t \MT\,\sym{\DGS\rgrad\Delta u}:(\DGS\rgrad\delta u) + 2t\MT\GTS:(\rgrad\Delta u^\top\rgrad\delta u),\\
	&D^2_{\bending}\Big(\frac{6}{t^3}\|\bending\|^2_{\MT^{-1}}\Big)[\delta\bending,\Delta \bending] = \frac{12}{t^3} \MT^{-1}\Delta\bending:\delta\bending,\\
	&D^2_u((\Hessian_\nu+(1-\rnv\cdot\nu)\rgrad\rnv):\bending)[\delta u,\Delta u]=(\rgrad^2 \delta u_i D_u\nu_i[\Delta u]+\rgrad^2\Delta u_i D_u\nu_i[\delta u]+\Hessian_{D^2_u\nu[\delta u,\Delta u]}\\
	&\qquad\qquad\qquad\qquad\qquad\qquad\qquad\qquad\qquad-\rnv\cdot D^2_u\nu[\delta u,\Delta u]\rgrad\rnv):\bending\\
	&D^2_{\bending\,u}((\Hessian_\nu+(1-\rnv\cdot\nu)\rgrad\rnv):\bending)[\delta u,\Delta \bending]= (\rgrad^2 \delta u_i \nu_i+\Hessian_{D_u\nu[\delta u]}-\rnv\cdot D_u\nu[\delta u]\rgrad\rnv):\Delta\bending.
\end{align*}
The second variation of the angle, $D^2_u(\sphericalangle(P_{\tau_E}^\perp(\Av{\nu}^n),\mu))[\delta u,\Delta u]$, follows by a straight-forward but lengthy application of product rules and are therefore omitted.
\subsection{Projection update in every Newton iteration:} 
If we average the normal vector after every Newton iteration there holds $\Av{\nu}^n\cdot\tau=0$ and thus
\begin{align*}
	&D_u(\sphericalangle(P_{\tau_E}^\perp(\Av{\nu}^n),\mu))[\delta u] =-\frac{1}{\sqrt{1-(\Av{\nu}^n\cdot\mu)^2}}\Av{\nu}^n\cdot D_u\mu[\delta u],\\
	&D^2_u(\sphericalangle(P_{\tau_E}^\perp(\Av{\nu}^n),\mu))[\delta u,\Delta u]=-\frac{(\Av{\nu}^n\cdot\mu)(\Av{\nu}^n\cdot D_u\mu[\Delta u] )(\Av{\nu}^n\cdot D_u\mu[\delta u])}{\sqrt{1-(\Av{\nu}^n\cdot\mu)^2}^3}\\
	&-\frac{1}{\sqrt{1-(\Av{\nu}^n\cdot\mu)^2}}\Big(\Av{\nu}^n\cdot D_u^2\mu[\delta u,\Delta u]+(\Av{\nu}^n\cdot\mu)(\Av{\nu}^n\cdot D_u\tau[\delta u])(\Av{\nu}^n\cdot D_u\tau[\Delta u])\Big).
\end{align*}

\end{document}

%% file: pre.tex
\usepackage{amsfonts,color}
\usepackage{amssymb,amscd}
\usepackage{footnote}
\usepackage[utf8]{inputenc}
\usepackage{mathtools}
\usepackage{amsthm,wasysym}
\usepackage[english]{babel}
\usepackage{bbm}
\usepackage{colonequals}
\usepackage{cancel}
\usepackage{dsfont}
\usepackage{adjustbox}
\usepackage{eucal}
\usepackage{graphicx}
\usepackage{bm}
\usepackage{amsfonts}
\usepackage{hyperref}
\usepackage{wrapfig}
\usepackage{subcaption}
\usepackage{amsmath}
\usepackage{float}
\usepackage{gensymb}
\usepackage{enumerate}
\usepackage{csvsimple}
\usepackage{stmaryrd}
\usepackage{accents}
\usepackage{pdfsync}

\usepackage{todonotes}

\usepackage{siunitx}
\newcolumntype{N}{c@{}S}

\usepackage{pgf,tikz,pgfplots,tikz-3dplot}
\pgfplotsset{compat=1.15}
\usepackage{mathrsfs}
\usetikzlibrary{arrows}
\usetikzlibrary{calc}
\usetikzlibrary{arrows.meta}
\usetikzlibrary{shapes}
\usetikzlibrary{decorations.markings}
\usetikzlibrary{automata}
\def\centerarc[#1](#2)(#3:#4:#5)% Syntax: [draw options] (center) (initial angle:final angle:radius)
{ \draw[#1] ($(#2)+({#5*cos(#3)},{#5*sin(#3)})$) arc (#3:#4:#5); }

\newcommand{\rotateRPY}[4][0/0/0]% point to be saved to \savedxyz, roll, pitch, yaw
{   \pgfmathsetmacro{\rollangle}{#2}
	\pgfmathsetmacro{\pitchangle}{#3}
	\pgfmathsetmacro{\yawangle}{#4}
	
	\pgfmathsetmacro{\newxx}{cos(\yawangle)*cos(\pitchangle)}% a
	\pgfmathsetmacro{\newxy}{sin(\yawangle)*cos(\pitchangle)}% d
	\pgfmathsetmacro{\newxz}{-sin(\pitchangle)}% g
	\path (\newxx,\newxy,\newxz);
	\pgfgetlastxy{\nxx}{\nxy};
	
	\pgfmathsetmacro{\newyx}{cos(\yawangle)*sin(\pitchangle)*sin(\rollangle)-sin(\yawangle)*cos(\rollangle)}% b
	\pgfmathsetmacro{\newyy}{sin(\yawangle)*sin(\pitchangle)*sin(\rollangle)+ cos(\yawangle)*cos(\rollangle)}% e
	\pgfmathsetmacro{\newyz}{cos(\pitchangle)*sin(\rollangle)}% h
	\path (\newyx,\newyy,\newyz);
	\pgfgetlastxy{\nyx}{\nyy};
	
	\pgfmathsetmacro{\newzx}{cos(\yawangle)*sin(\pitchangle)*cos(\rollangle)+ sin(\yawangle)*sin(\rollangle)}
	\pgfmathsetmacro{\newzy}{sin(\yawangle)*sin(\pitchangle)*cos(\rollangle)-cos(\yawangle)*sin(\rollangle)}
	\pgfmathsetmacro{\newzz}{cos(\pitchangle)*cos(\rollangle)}
	\path (\newzx,\newzy,\newzz);
	\pgfgetlastxy{\nzx}{\nzy};
	
	\foreach \x/\y/\z in {#1}
	{   \pgfmathsetmacro{\transformedx}{\x*\newxx+\y*\newyx+\z*\newzx}
		\pgfmathsetmacro{\transformedy}{\x*\newxy+\y*\newyy+\z*\newzy}
		\pgfmathsetmacro{\transformedz}{\x*\newxz+\y*\newyz+\z*\newzz}

	}
}

\tikzset{RPY/.style={x={(\nxx,\nxy)},y={(\nyx,\nyy)},z={(\nzx,\nzy)}}}

\usepackage{chngcntr}
\counterwithin{equation}{section}

\newtheorem{theorem}{Theorem}[section]
\newtheorem{lemma}[theorem]{Lemma}

\theoremstyle{definition}

\theoremstyle{remark}
\newtheorem{remark}[theorem]{Remark}

\allowdisplaybreaks[1]

\def\d{\partial}

\newcommand{\re}[1]{\hat{#1}}
\newcommand{\de}[1]{{#1}}

\newcommand{\T}{\ensuremath{\re{\mathscr{T}}}} % Triangulation
\newcommand{\E}{\ensuremath{\re{\mathscr{E}}}} % Skeleton
\newcommand{\dT}{\ensuremath{\de{\mathscr{T}}}} % Triangulation
\newcommand{\spaceU}{\ensuremath{\mathcal{U}}}
\newcommand{\spaceB}{\ensuremath{\Sigma}}
\newcommand{\spaceS}{\ensuremath{\mathcal{N}}}
\newcommand{\spaceH}{\ensuremath{\Lambda}}

\renewcommand{\L}{\ensuremath{\mathscr{L}}} % Lagrange functional

\newcommand{\mat}[1]{\ensuremath{\begin{pmatrix}#1\end{pmatrix}}}	% matrix

\newcommand{\R}{\ensuremath{\mathbb{R}}}

\newcommand{\tcross}{\boldsymbol{\times}}
\newcommand{\dx}{\,\mathrm{dx}}
\newcommand{\ds}{\,\mathrm{ds}}

\newcommand{\refT}{\tilde{T}}
\newcommand{\refE}{\tilde{E}}
\newcommand{\refPhi}{\tilde{\Phi}}

\newcommand{\veps}{\varepsilon}
\newcommand{\shearModulus}{G}
\newcommand{\shearCorrection}{\kappa}

\newcommand{\bending}{\bm{\sigma}}
\renewcommand{\angle}{\alpha}
\newcommand{\curv}{\bm{\kappa}}

\newcommand{\Av}[1]{\{#1\}}
\newcommand{\jump}[2][]{\ensuremath{\ifthenelse{\equal{#1}{}}{\llbracket #2 \rrbracket}{\llbracket #2 \rrbracket}_{#1}}}
\newcommand{\jmp}[1]{\ensuremath{\llbracket #1 \rrbracket}}
\newcommand{\mean}[1]{\ensuremath{\{\!\{#1\}\!\}}}

\newcommand{\tr}[1]{\ensuremath{\,\mathrm{tr}(#1)}}
\newcommand{\sym}[1]{\ensuremath{\,\mathrm{sym}(#1)}}
\renewcommand{\skew}[1]{\ensuremath{\,\mathrm{skew}(#1)}}
\newcommand{\cof}[1]{\ensuremath{\,\mathrm{cof}(#1)}}

\newcommand{\norm}[2][]{\ensuremath{\ifthenelse{\equal{#1}{}}{\left\|#2\right\|}{\left\|#2\right\|_{#1}}}}

\renewcommand{\forall}{\text{ for all }}

\newcommand{\rSurf}{\re{\mathcal{S}}}
\newcommand{\rnv}{\re{\nu}}
\newcommand{\rtv}{\re{\tau}}
\newcommand{\rcnv}{\re{\mu}}
\newcommand{\rx}{\re{x}}
\newcommand{\rshear}{\re{\gamma}}

\newcommand{\Hessian}{\mathcal{H}}

\newcommand{\rgrad}{\nabla_{\rSurf}}
\newcommand{\rcovder}{\nabla_{\rSurf}^{\mathrm{cov}}}
\newcommand{\rhesse}{\nabla_{\rSurf}^2}

\newcommand{\cphi}{\circ\phi}

\newcommand{\dSurf}{\de{\mathcal{S}}}
\newcommand{\dnv}[1][]{\ifthenelse{\equal{#1}{}}{\de{\nu}\cphi}{\de{\nu}_{#1}\cphi}}
\newcommand{\dtv}[1][]{\ifthenelse{\equal{#1}{}}{\de{\tau}\cphi}{\de{\tau}_{#1}\cphi}}
\newcommand{\dcnv}[1][]{\ifthenelse{\equal{#1}{}}{\de{\mu}\cphi}{\de{\mu}_{#1}\cphi}}
\newcommand{\dshear}[1][]{\ifthenelse{\equal{#1}{}}{\de{\gamma}\cphi}{\de{\gamma}_{#1}\cphi}}
\newcommand{\director}[1][]{\ifthenelse{\equal{#1}{}}{\tilde{\nu}\cphi}{\tilde{\nu}_{#1}\cphi}}

\newcommand{\MT}{\mathbb{M}}
\newcommand{\DG}{\bm{F}}

\newcommand{\DGS}{{\bm{F}_{\rSurf}}}
\newcommand{\CGTS}{{\bm{C}_{\rSurf}}}
\newcommand{\GTS}{{\bm{E}_{\rSurf}}}

\newcommand{\PoissonRatio}{{\bar{\nu}}}
\newcommand{\YoungsModulus}{{\bar{E}}}

\newcommand{\bA}{\bm{A}}
\newcommand{\bB}{\bm{B}}
\newcommand{\bR}{\bm{R}}
\newcommand{\bI}{\bm{I}}
\newcommand{\bH}{\bm{H}}

\newcommand{\Proj}{{\bm{P}_{\rSurf}}}

\newcommand{\Pol}{\ensuremath{\mathcal{P}}}
\newcommand{\RegInt}[1][]{\mathcal{I}_{h}^{\Regge\ifthenelse{\equal{#1}{}}{}{,#1}}}
\newcommand{\HoneInt}[1][]{\mathcal{I}_{h}^{\VV,\ifthenelse{\equal{#1}{}}{}{#1}}}

\newcommand{\Regge}[1][]{\ensuremath{\mathcal{R}\ifthenelse{\equal{#1}{}}{}{(#1)}}}

\newcommand{\SO}[1]{\mathrm{SO}(#1)}
\newcommand{\Stwo}{\mathbb{S}^2}

\renewcommand{\div}{\mathrm{div}}

\newcommand{\Ltwo}[1][]{\ensuremath{L^2\ifthenelse{\equal{#1}{}}{}{(#1)}}}
\newcommand{\Winf}[1][]{\ensuremath{W^{1,\infty}\ifthenelse{\equal{#1}{}}{}{(#1)}}}
\newcommand{\Hone}[1][]{\ensuremath{H^1\ifthenelse{\equal{#1}{}}{}{(#1)}}}
\newcommand{\Honeh}[1][]{\ensuremath{H^1_h\ifthenelse{\equal{#1}{}}{}{(#1)}}}
\newcommand{\Honez}[1][]{\ensuremath{H^1_0\ifthenelse{\equal{#1}{}}{}{(#1)}}}
\newcommand{\Hmone}[1][]{\ensuremath{H^{-1}\ifthenelse{\equal{#1}{}}{}{(#1)}}}
\newcommand{\Htwo}[1][]{\ensuremath{H^2\ifthenelse{\equal{#1}{}}{}{(#1)}}}
\newcommand{\Htwoz}[1][]{\ensuremath{H^2_0\ifthenelse{\equal{#1}{}}{}{(#1)}}}
\newcommand{\Hmtwo}[1][]{\ensuremath{H^{-2}\ifthenelse{\equal{#1}{}}{}{(#1)}}}
\newcommand{\HDiv}[1][]{\ensuremath{H(\mathrm{div}\ifthenelse{\equal{#1}{}}{}{,#1})}}
\newcommand{\HDivz}[1][]{\ensuremath{H_0(\mathrm{div}\ifthenelse{\equal{#1}{}}{}{,#1})}}
\newcommand{\HCurl}[1][]{\ensuremath{H(\mathrm{curl}\ifthenelse{\equal{#1}{}}{}{,#1})}}
\newcommand{\HDivDiv}[1][]{\ensuremath{H(\mathrm{divdiv}\ifthenelse{\equal{#1}{}}{}{,#1})}}
\newcommand{\HCurlCurl}[1][]{\ensuremath{H( \mathrm{curl curl}\ifthenelse{\equal{#1}{}}{}{,#1})}}
\newcommand{\Cone}[1][]{\ensuremath{C^1\ifthenelse{\equal{#1}{}}{}{(#1)}}}
\newcommand{\Czero}[1][]{\ensuremath{C^0\ifthenelse{\equal{#1}{}}{}{(#1)}}}
\newcommand{\Ctwo}[1][]{\ensuremath{C^2\ifthenelse{\equal{#1}{}}{}{(#1)}}}

\newcommand{\Cinf}[1][]{\ensuremath{C^{\infty}\ifthenelse{\equal{#1}{}}{}{(#1)}}}
\newcommand{\DF}[1][]{\ensuremath{C_0^{\infty,*}\ifthenelse{\equal{#1}{}}{}{(#1)}}}

\newcommand{\idop}{\mathrm{id}}

%% file: results_slit_annular.tex
	\resizebox{0.48\textwidth}{!}{
	\begin{tikzpicture}
	\begin{axis}[
		legend style={at={(0,1)}, anchor=north west},
		xlabel={displacement},
		ylabel={$P/P_{\mathrm{max}}$},
		]
		
\addlegendentry{Sze et al. 2004}
\addplot[color=black, mark=none] coordinates {
	( 0 , 0.0 )
	( 1.789 , 0.025 )
	( 3.37 , 0.05 )
	( 4.72 , 0.075 )
	( 5.876 , 0.1 )
	( 6.872 , 0.125 )
	( 7.736 , 0.15 )
	( 9.16 , 0.2 )
	( 10.288 , 0.25 )
	( 11.213 , 0.3 )
	( 11.992 , 0.35 )
	( 12.661 , 0.4 )
	( 13.247 , 0.45 )
	( 13.768 , 0.5 )
	( 14.24 , 0.55 )
	( 14.674 , 0.6 )
	( 15.081 , 0.65 )
	( 15.469 , 0.7 )
	( 15.842 , 0.75 )
	( 16.202 , 0.8 )
	( 16.55 , 0.85 )
	( 16.886 , 0.9 )
	( 17.212 , 0.95 )
	( 17.528 , 1.0 )
};
	
	\addplot[color=black, mark=none, forget plot] coordinates {
	( 0 , 0.0 )
	( 1.305 , 0.025 )
	( 2.455 , 0.05 )
	( 3.435 , 0.075 )
	( 4.277 , 0.1 )
	( 5.007 , 0.125 )
	( 5.649 , 0.15 )
	( 6.725 , 0.2 )
	( 7.602 , 0.25 )
	( 8.34 , 0.3 )
	( 8.974 , 0.35 )
	( 9.529 , 0.4 )
	( 10.023 , 0.45 )
	( 10.468 , 0.5 )
	( 10.876 , 0.55 )
	( 11.257 , 0.6 )
	( 11.62 , 0.65 )
	( 11.97 , 0.7 )
	( 12.31 , 0.75 )
	( 12.642 , 0.8 )
	( 12.966 , 0.85 )
	( 13.282 , 0.9 )
	( 13.59 , 0.95 )
	( 13.891 , 1.0 )
};
	
	\addlegendentry{HHJ $h=2$}
	\addplot[color=orange, mark=o, mark options={scale=0.5}] coordinates {
	( 0 , 0.0 )
	( 3.2540201967365787 , 0.05 )
	( 5.759685178332613 , 0.1 )
	( 7.71443332987722 , 0.15 )
	( 9.220894802826166 , 0.2 )
	( 10.40320373676282 , 0.25 )
	( 11.357773668266649 , 0.3 )
	( 12.149405416522205 , 0.35 )
	( 12.833835134655825 , 0.4 )
	( 13.451418358003272 , 0.45 )
	( 14.004188539694919 , 0.5 )
	( 14.506360941792595 , 0.55 )
	( 14.968895419452705 , 0.6 )
	( 15.39978921417839 , 0.65 )
	( 15.804694021369126 , 0.7 )
	( 16.187709671110724 , 0.75 )
	( 16.55195207797078 , 0.8 )
	( 16.899858655655805 , 0.85 )
	( 17.233369155484258 , 0.9 )
	( 17.554046972648575 , 0.95 )
	( 17.86316587575803 , 1.0 )
};
	
\addplot[color=orange, mark=o, mark options={scale=0.5}, forget plot] coordinates {	
	( 0 , 0.0 )
	( 2.357408691324333 , 0.05 )
	( 4.167512860066901 , 0.1 )
	( 5.6143094984462465 , 0.15 )
	( 6.7617066796480625 , 0.2 )
	( 7.685831629830488 , 0.25 )
	( 8.448555896140363 , 0.3 )
	( 9.093236630596937 , 0.35 )
	( 9.66239597388032 , 0.4 )
	( 10.189161010532985 , 0.45 )
	( 10.668572647552205 , 0.5 )
	( 11.110445863923145 , 0.55 )
	( 11.522817895950702 , 0.6 )
	( 11.911525717291905 , 0.65 )
	( 12.280580223560142 , 0.7 )
	( 12.632844104215327 , 0.75 )
	( 12.97049452982366 , 0.8 )
	( 13.29524504480578 , 0.85 )
	( 13.608468148947049 , 0.9 )
	( 13.911276386883898 , 0.95 )
	( 14.204580756048022 , 1.0 )
};

	\addlegendentry{HHJ $h=0.25$}
	\addplot[color=red, mark=x, mark options={scale=1}] coordinates {	
	( 0 , 0.0 )
	( 3.3569413066066556 , 0.05 )
	( 5.849358958786691 , 0.1 )
	( 7.705858588235633 , 0.15 )
	( 9.127710449832135 , 0.2 )
	( 10.25453692507849 , 0.25 )
	( 11.17695573841287 , 0.3 )
	( 11.952137656008942 , 0.35 )
	( 12.617527779640296 , 0.4 )
	( 13.199089805796323 , 0.45 )
	( 13.715917599937054 , 0.5 )
	( 14.1830510994187 , 0.55 )
	( 14.613231438409896 , 0.6 )
	( 15.017191169593744 , 0.65 )
	( 15.402279068462729 , 0.7 )
	( 15.772139382890819 , 0.75 )
	( 16.12855654826406 , 0.8 )
	( 16.472756813133262 , 0.85 )
	( 16.80565238295157 , 0.9 )
	( 17.12787399860592 , 0.95 )
	( 17.439851973610967 , 1.0 )
};

\addplot[color=red, mark=x, mark options={scale=1}, forget plot] coordinates {
	( 0 , 0.0 )
	( 2.444975472500222 , 0.05 )
	( 4.254739691156199 , 0.1 )
	( 5.6221122602165 , 0.15 )
	( 6.696172443357651 , 0.2 )
	( 7.571002147027235 , 0.25 )
	( 8.305144007400392 , 0.3 )
	( 8.935304521545543 , 0.35 )
	( 9.486058395440898 , 0.4 )
	( 9.97506693243977 , 0.45 )
	( 10.415876507651504 , 0.5 )
	( 10.819711257958438 , 0.55 )
	( 11.196653514843211 , 0.6 )
	( 11.555579277374681 , 0.65 )
	( 11.902467257052677 , 0.7 )
	( 12.239820044028063 , 0.75 )
	( 12.568440808804176 , 0.8 )
	( 12.88874571948002 , 0.85 )
	( 13.200989985449366 , 0.9 )
	( 13.505269879197128 , 0.95 )
	( 13.801579244503525 , 1.0 )
};
	\end{axis}
	\end{tikzpicture}}
\resizebox{0.48\textwidth}{!}{\begin{tikzpicture}
	\begin{axis}[
	legend style={at={(0,1)}, anchor=north west},
	xlabel={displacement},
	ylabel={$P/P_{\mathrm{max}}$},
	]
	\addlegendentry{Sze et al. 2004}
	\addplot[color=black, mark=none] coordinates {
		( 0 , 0.0 )
		( 1.789 , 0.025 )
		( 3.37 , 0.05 )
		( 4.72 , 0.075 )
		( 5.876 , 0.1 )
		( 6.872 , 0.125 )
		( 7.736 , 0.15 )
		( 9.16 , 0.2 )
		( 10.288 , 0.25 )
		( 11.213 , 0.3 )
		( 11.992 , 0.35 )
		( 12.661 , 0.4 )
		( 13.247 , 0.45 )
		( 13.768 , 0.5 )
		( 14.24 , 0.55 )
		( 14.674 , 0.6 )
		( 15.081 , 0.65 )
		( 15.469 , 0.7 )
		( 15.842 , 0.75 )
		( 16.202 , 0.8 )
		( 16.55 , 0.85 )
		( 16.886 , 0.9 )
		( 17.212 , 0.95 )
		( 17.528 , 1.0 )
	};
	
	\addplot[color=black, mark=none, forget plot] coordinates {
		( 0 , 0.0 )
		( 1.305 , 0.025 )
		( 2.455 , 0.05 )
		( 3.435 , 0.075 )
		( 4.277 , 0.1 )
		( 5.007 , 0.125 )
		( 5.649 , 0.15 )
		( 6.725 , 0.2 )
		( 7.602 , 0.25 )
		( 8.34 , 0.3 )
		( 8.974 , 0.35 )
		( 9.529 , 0.4 )
		( 10.023 , 0.45 )
		( 10.468 , 0.5 )
		( 10.876 , 0.55 )
		( 11.257 , 0.6 )
		( 11.62 , 0.65 )
		( 11.97 , 0.7 )
		( 12.31 , 0.75 )
		( 12.642 , 0.8 )
		( 12.966 , 0.85 )
		( 13.282 , 0.9 )
		( 13.59 , 0.95 )
		( 13.891 , 1.0 )
	};

\addlegendentry{TDNNS $h=2$}
\addplot[color=orange, mark=o, mark options={scale=0.5}] coordinates {	
( 0 , 0 )
( 2.916569269958399 , 0.05 )
( 5.478928582970616 , 0.1 )
( 7.700633287740316 , 0.15 )
( 9.221293107740129 , 0.2 )
( 10.40363141583161 , 0.25 )
( 11.358220359621837 , 0.3 )
( 12.149867493294845 , 0.35 )
( 12.83434673787563 , 0.4 )
( 13.451958257080769 , 0.45 )
( 14.004750455168105 , 0.5 )
( 14.506947332108718 , 0.55 )
( 14.96950861222615 , 0.6 )
( 15.400429826189193 , 0.65 )
( 15.805361606359543 , 0.7 )
( 16.18840359022511 , 0.75 )
( 16.552671697639713 , 0.8 )
( 16.900603340214317 , 0.85 )
( 17.234138242666273 , 0.9 )
( 17.554839759407226 , 0.95 )
( 17.863981613249152 , 1.0 )
};

\addplot[color=orange, mark=o, mark options={scale=0.5}, forget plot] coordinates {	
( 0 , 0 )
( 2.143182788318975 , 0.05 )
( 3.993713638977486 , 0.1 )
( 5.605740509598568 , 0.15 )
( 6.762069735092866 , 0.2 )
( 7.686227893506629 , 0.25 )
( 8.44897455965427 , 0.3 )
( 9.093673548735852 , 0.35 )
( 9.662886575262563 , 0.4 )
( 10.189684867796734 , 0.45 )
( 10.669121850751017 , 0.5 )
( 11.111022495297707 , 0.55 )
( 11.52342390605545 , 0.6 )
( 11.91216125273461 , 0.65 )
( 12.281244465605033 , 0.7 )
( 12.633536156948512 , 0.75 )
( 12.971213572626848 , 0.8 )
( 13.295990296357333 , 0.85 )
( 13.609238828361198 , 0.9 )
( 13.912071691843447 , 0.95 )
( 14.205399853193354 , 1.0 )
};

\addlegendentry{TDNNS $h=0.25$}
\addplot[color=red, mark=x, mark options={scale=1}] coordinates {	
( 0 , 0 )
(3.16060721519766, 0.05)
(5.837183598398284, 0.1)
(7.694360245540096, 0.15)
(9.130515049446435, 0.2)
(10.257526789485464, 0.25)
(11.180076259001446, 0.3)
(11.95534240691727, 0.35)
(12.620783987155315, 0.4)
(13.20237975109802, 0.45)
(13.719239404032292, 0.5)
(14.186421404838798, 0.55)
(14.616686460874325, 0.6)
(15.020772063981479, 0.65)
(15.406005294073344, 0.7)
(15.776008454320332, 0.75)
(16.13256022977733, 0.8)
(16.476886593142652, 0.85)
(16.809898868475415, 0.9)
(17.132226752423037, 0.95)
(17.444300024746155, 1)
};

\addplot[color=red, mark=x, mark options={scale=1}, forget plot] coordinates {	
	( 0 , 0 )
	(2.315426006836379, 0.05)
	(4.247942465593096, 0.1)
	(5.615685256943746, 0.15)
	(6.698778666553625, 0.2)
	(7.57381926226578, 0.25)
	(8.308110035634853, 0.3)
	(8.938367891420382, 0.35)
	(9.489183613974939, 0.4)
	(9.978235001437406, 0.45)
	(10.419085059423223, 0.5)
	(10.822977138041699, 0.55)
	(11.200014147528098, 0.6)
	(11.559077930774224, 0.65)
	(11.906124945843139, 0.7)
	(12.243634315766682, 0.75)
	(12.572402581920123, 0.8)
	(12.892845388441536, 0.85)
	(13.205217006636552, 0.9)
	(13.509612681532737, 0.95)
	(13.806025792566581, 1)
};

\end{axis}
\end{tikzpicture}}

%% file: results_pinched_cyl.tex
	\resizebox{0.48\textwidth}{!}{
	\begin{tikzpicture}
	\begin{axis}[
		legend style={at={(0,1)}, anchor=north west},
		xlabel={displacement},
		ylabel={$P/P_{\mathrm{max}}$},
		]
\addlegendentry{Sze et al. 2004}
\addplot[color=black, mark=none] coordinates {
	( 0 , 0.0 )
	( 5.421 , 0.05 )
	( 16.100 , 0.1 )
	( 22.195 , 0.125 )
	( 27.657 , 0.15 )
	( 32.700 , 0.175 )
	( 37.582 , 0.2 )
	( 42.633 , 0.225 )
	( 48.537 , 0.25 )
	( 56.355 , 0.275 )
	( 66.410 , 0.3 )
	( 79.810 , 0.325 )
	( 94.669 , 0.35 )
	( 113.704 , 0.4 )
	( 124.751 , 0.45 )
	( 132.653 , 0.5 )
	( 138.920 , 0.55 )
	( 144.185 , 0.6 )
	( 148.770 , 0.65 )
	( 152.863 , 0.7 )
	( 156.584 , 0.75 )
	( 160.015 , 0.8 )
	( 163.211 , 0.85 )
	( 166.200 , 0.9 )
	( 168.973 , 0.95 )
	( 171.505 , 1.0 )
};

\addlegendentry{HHJ p1 $16\times16$}
\addplot[color=blue, mark=o, mark options={scale=0.5}] coordinates {
( 0 , 0 )
( 3.29997 , 0.033 )
( 8.59679 , 0.067 )
( 16.15911 , 0.1 )
( 23.70729 , 0.133 )
( 30.66666 , 0.167 )
( 37.18093 , 0.2 )
( 44.53333 , 0.233 )
( 54.91291 , 0.267 )
( 67.99161 , 0.3 )
( 83.14401 , 0.333 )
( 89.31763 , 0.367 )
( 103.82739 , 0.4 )
( 107.96148 , 0.433 )
( 120.27468 , 0.467 )
( 123.12684 , 0.5 )
( 125.55681 , 0.533 )
( 127.6755 , 0.567 )
( 137.13027 , 0.6 )
( 138.81218 , 0.633 )
( 140.32305 , 0.667 )
( 141.69328 , 0.7 )
( 142.95168 , 0.733 )
( 144.13604 , 0.767 )
( 152.59159 , 0.8 )
( 152.99979 , 0.833 )
( 154.0009 , 0.867 )
( 154.94238 , 0.9 )
( 155.83763 , 0.933 )
( 156.7059 , 0.967 )
( 163.65756 , 1.0 )
};

\addlegendentry{HHJ p2 $16\times16$}
\addplot[color=red, mark=x, mark options={scale=1}] coordinates {	
( 0 , 0 )
( 3.21526 , 0.033 )
( 8.12435 , 0.067 )
( 15.81922 , 0.1 )
( 23.76736 , 0.133 )
( 30.81622 , 0.167 )
( 37.50117 , 0.2 )
( 44.64956 , 0.233 )
( 53.95322 , 0.267 )
( 67.0362 , 0.3 )
( 86.27827 , 0.333 )
( 102.86969 , 0.367 )
( 114.1751 , 0.4 )
( 121.823 , 0.433 )
( 127.88912 , 0.467 )
( 132.43156 , 0.5 )
( 137.6463 , 0.533 )
( 140.95344 , 0.567 )
( 144.08759 , 0.6 )
( 148.22226 , 0.633 )
( 150.78136 , 0.667 )
( 153.06908 , 0.7 )
( 155.90652 , 0.733 )
( 158.895 , 0.767 )
( 160.84692 , 0.8 )
( 162.66394 , 0.833 )
( 164.67041 , 0.867 )
( 167.57799 , 0.9 )
( 169.24834 , 0.933 )
( 170.75256 , 0.967 )
( 172.20456 , 1.0 )
};

\addlegendentry{HHJ p1 $32\times32$}
\addplot[color=cyan, mark=diamond, mark options={scale=1}] coordinates {
( 0 , 0 )
( 3.24252 , 0.033 )
( 8.22893 , 0.067 )
( 15.9063 , 0.1 )
( 23.75457 , 0.133 )
( 30.74688 , 0.167 )
( 37.38397 , 0.2 )
( 44.49529 , 0.233 )
( 53.75914 , 0.267 )
( 66.45633 , 0.3 )
( 84.28864 , 0.333 )
( 101.77589 , 0.367 )
( 113.28717 , 0.4 )
( 121.35363 , 0.433 )
( 127.57743 , 0.467 )
( 132.68324 , 0.5 )
( 137.04318 , 0.533 )
( 140.87372 , 0.567 )
( 144.29886 , 0.6 )
( 147.41385 , 0.633 )
( 150.2787 , 0.667 )
( 152.93551 , 0.7 )
( 155.42405 , 0.733 )
( 157.7644 , 0.767 )
( 159.98193 , 0.8 )
( 162.0921 , 0.833 )
( 164.10223 , 0.867 )
( 166.03134 , 0.9 )
( 167.86543 , 0.933 )
( 169.62059 , 0.967 )
( 171.27879 , 1.0 )
};

\addlegendentry{HHJ p2 $32\times32$}
\addplot[color=orange, mark=square, mark options={scale=1}] coordinates {	
( 0 , 0 )
( 3.20059 , 0.033 )
( 8.10748 , 0.067 )
( 15.85925 , 0.1 )
( 23.82996 , 0.133 )
( 30.88055 , 0.167 )
( 37.55515 , 0.2 )
( 44.68945 , 0.233 )
( 54.02178 , 0.267 )
( 67.20833 , 0.3 )
( 86.19892 , 0.333 )
( 103.24033 , 0.367 )
( 114.07041 , 0.4 )
( 121.74594 , 0.433 )
( 127.7457 , 0.467 )
( 132.72101 , 0.5 )
( 137.00616 , 0.533 )
( 140.79354 , 0.567 )
( 144.2036 , 0.6 )
( 147.31886 , 0.633 )
( 150.1951 , 0.667 )
( 152.87568 , 0.7 )
( 155.39146 , 0.733 )
( 157.7681 , 0.767 )
( 160.02369 , 0.8 )
( 162.17364 , 0.833 )
( 164.22518 , 0.867 )
( 166.18335 , 0.9 )
( 168.0435 , 0.933 )
( 169.80195 , 0.967 )
( 171.45475 , 1.0 )
};

\end{axis}
\end{tikzpicture}}
	\resizebox{0.48\textwidth}{!}{
	\begin{tikzpicture}
		\begin{axis}[
			legend style={at={(0,1)}, anchor=north west},
			xlabel={displacement},
			ylabel={$P/P_{\mathrm{max}}$},
			]
			\addlegendentry{Sze et al. 2004}
			\addplot[color=black, mark=none] coordinates {
				( 0 , 0.0 )
				( 5.421 , 0.05 )
				( 16.100 , 0.1 )
				( 22.195 , 0.125 )
				( 27.657 , 0.15 )
				( 32.700 , 0.175 )
				( 37.582 , 0.2 )
				( 42.633 , 0.225 )
				( 48.537 , 0.25 )
				( 56.355 , 0.275 )
				( 66.410 , 0.3 )
				( 79.810 , 0.325 )
				( 94.669 , 0.35 )
				( 113.704 , 0.4 )
				( 124.751 , 0.45 )
				( 132.653 , 0.5 )
				( 138.920 , 0.55 )
				( 144.185 , 0.6 )
				( 148.770 , 0.65 )
				( 152.863 , 0.7 )
				( 156.584 , 0.75 )
				( 160.015 , 0.8 )
				( 163.211 , 0.85 )
				( 166.200 , 0.9 )
				( 168.973 , 0.95 )
				( 171.505 , 1.0 )
			};

			\addlegendentry{TDNNS p1 $16\times16$}
			\addplot[color=blue, mark=o, mark options={scale=0.5}] coordinates {
				( 0 , 0 )
				( 3.33455 , 0.033 )
				( 8.71529 , 0.067 )
				( 16.26064 , 0.1 )
				( 23.7813 , 0.133 )
				( 30.7487 , 0.167 )
				( 37.28633 , 0.2 )
				( 44.70988 , 0.233 )
				( 55.19861 , 0.267 )
				( 68.29562 , 0.3 )
				( 75.34301 , 0.333 )
				( 89.41408 , 0.367 )
				( 97.4297 , 0.4 )
				( 108.07895 , 0.433 )
				( 111.49772 , 0.467 )
				( 123.28669 , 0.5 )
				( 125.69201 , 0.533 )
				( 127.77275 , 0.567 )
				( 129.61371 , 0.6 )
				( 139.63763 , 0.633 )
				( 141.05036 , 0.667 )
				( 142.32979 , 0.7 )
				( 143.49156 , 0.733 )
				( 144.54697 , 0.767 )
				( 145.49599 , 0.8 )
				( 145.90553 , 0.833 )
				( 143.27475 , 0.867 )
				( 160.65619 , 0.9 )
				( 162.00829 , 0.933 )
				( 162.95251 , 0.967 )
				( 163.83499 , 1.0 )
			};
			
			\addlegendentry{TDNNS p2 $16\times16$}
			\addplot[color=red, mark=x, mark options={scale=1}] coordinates {
				( 0 , 0 )
				( 3.27178 , 0.033 )
				( 8.28247 , 0.067 )
				( 15.99927 , 0.1 )
				( 23.89608 , 0.133 )
				( 30.93494 , 0.167 )
				( 37.63982 , 0.2 )
				( 44.85496 , 0.233 )
				( 54.25443 , 0.267 )
				( 67.45117 , 0.3 )
				( 86.75191 , 0.333 )
				( 103.20696 , 0.367 )
				( 114.48232 , 0.4 )
				( 122.26015 , 0.433 )
				( 128.14725 , 0.467 )
				( 132.68976 , 0.5 )
				( 137.93068 , 0.533 )
				( 141.1848 , 0.567 )
				( 144.32647 , 0.6 )
				( 148.51201 , 0.633 )
				( 151.02142 , 0.667 )
				( 153.29115 , 0.7 )
				( 156.45975 , 0.733 )
				( 159.16544 , 0.767 )
				( 161.08701 , 0.8 )
				( 162.89049 , 0.833 )
				( 164.9296 , 0.867 )
				( 167.86199 , 0.9 )
				( 169.49969 , 0.933 )
				( 170.99105 , 0.967 )
				( 172.4382 , 1.0 )	
			};
			
			\addlegendentry{TDNNS p1 $32\times32$}
			\addplot[color=cyan, mark=diamond, mark options={scale=1}] coordinates {
				( 0 , 0 )
				( 3.2924 , 0.033 )
				( 8.37387 , 0.067 )
				( 16.06157 , 0.1 )
				( 23.86326 , 0.133 )
				( 30.84873 , 0.167 )
				( 37.50763 , 0.2 )
				( 44.68925 , 0.233 )
				( 54.05804 , 0.267 )
				( 66.84414 , 0.3 )
				( 84.79574 , 0.333 )
				( 102.19749 , 0.367 )
				( 113.63399 , 0.4 )
				( 121.66366 , 0.433 )
				( 127.86717 , 0.467 )
				( 132.95992 , 0.5 )
				( 137.31025 , 0.533 )
				( 141.13692 , 0.567 )
				( 144.5578 , 0.6 )
				( 147.66906 , 0.633 )
				( 150.5339 , 0.667 )
				( 153.18697 , 0.7 )
				( 155.67704 , 0.733 )
				( 158.01429 , 0.767 )
				( 160.23302 , 0.8 )
				( 162.34147 , 0.833 )
				( 164.35028 , 0.867 )
				( 166.28 , 0.9 )
				( 168.10964 , 0.933 )
				( 169.86691 , 0.967 )
				( 171.52297 , 1.0 )
			};

			\addlegendentry{TDNNS p2 $32\times32$}
			\addplot[color=orange, mark=square, mark options={scale=1}] coordinates {	
				( 0 , 0 )
				( 3.27953 , 0.033 )
				( 8.3302 , 0.067 )
				( 16.12318 , 0.1 )
				( 24.0195 , 0.133 )
				( 31.04539 , 0.167 )
				( 37.73031 , 0.2 )
				( 44.92365 , 0.233 )
				( 54.35875 , 0.267 )
				( 67.64256 , 0.3 )
				( 86.75282 , 0.333 )
				( 103.69525 , 0.367 )
				( 114.45709 , 0.4 )
				( 122.09907 , 0.433 )
				( 128.08021 , 0.467 )
				( 133.04421 , 0.5 )
				( 137.32215 , 0.533 )
				( 141.10493 , 0.567 )
				( 144.5119 , 0.6 )
				( 147.62519 , 0.633 )
				( 150.50027 , 0.667 )
				( 153.18003 , 0.7 )
				( 155.69543 , 0.733 )
				( 158.0716 , 0.767 )
				( 160.32675 , 0.8 )
				( 162.47582 , 0.833 )
				( 164.52611 , 0.867 )
				( 166.48268 , 0.9 )
				( 168.34112 , 0.933 )
				( 170.09873 , 0.967 )
				( 171.75175 , 1.0 )
			};
			
		\end{axis}
\end{tikzpicture}}

%% file: results_t_struct.tex
	\resizebox{0.48\textwidth}{!}{
		\begin{tikzpicture}
			\begin{axis}[
				legend style={at={(0,1)}, anchor=north west},
				xlabel={displacement},
				ylabel={$P/P_{\mathrm{max}}$},
				]
				
		\addlegendentry{HHJ h0.5}
	\addplot[color=blue, mark=square,mark options={scale=0.5}] coordinates {
	( 0 , 0 )
	( 0.19874 , 0.05 )
	( 0.40867 , 0.1 )
	( 0.58559 , 0.15 )
	( 0.72109 , 0.2 )
	( 0.82348 , 0.25 )
	( 0.90214 , 0.3 )
	( 0.96405 , 0.35 )
	( 1.01396 , 0.4 )
	( 1.05512 , 0.45 )
	( 1.08976 , 0.5 )
	( 1.11943 , 0.55 )
	( 1.14526 , 0.6 )
	( 1.16807 , 0.65 )
	( 1.18846 , 0.7 )
	( 1.20691 , 0.75 )
	( 1.22377 , 0.8 )
	( 1.23932 , 0.85 )
	( 1.25376 , 0.9 )
	( 1.26729 , 0.95 )
	( 1.28002 , 1.0 )
		
	};
	
	\addlegendentry{HHJ h0.25}
	\addplot[color=red, mark=*,mark options={scale=0.5}] coordinates {
		( 0 , 0 )
		( 0.19874 , 0.05 )
		( 0.40861 , 0.1 )
		( 0.58539 , 0.15 )
		( 0.72068 , 0.2 )
		( 0.82277 , 0.25 )
		( 0.90099 , 0.3 )
		( 0.96228 , 0.35 )
		( 1.0114 , 0.4 )
		( 1.05158 , 0.45 )
		( 1.08505 , 0.5 )
		( 1.11336 , 0.55 )
		( 1.13765 , 0.6 )
		( 1.15874 , 0.65 )
		( 1.17724 , 0.7 )
		( 1.19363 , 0.75 )
		( 1.20828 , 0.8 )
		( 1.22146 , 0.85 )
		( 1.2334 , 0.9 )
		( 1.2443 , 0.95 )
		( 1.25429 , 1.0 )
		
	};

\addlegendentry{TDNNS  h0.5}
\addplot[color=magenta, mark=diamond,mark options={scale=0.5}] coordinates {
( 0 , 0 )
( 0.19933 , 0.05 )
( 0.41041 , 0.1 )
( 0.58869 , 0.15 )
( 0.72563 , 0.2 )
( 0.82954 , 0.25 )
( 0.90979 , 0.3 )
( 0.97339 , 0.35 )
( 1.02509 , 0.4 )
( 1.06812 , 0.45 )
( 1.1047 , 0.5 )
( 1.13637 , 0.55 )
( 1.16425 , 0.6 )
( 1.18914 , 0.65 )
( 1.21162 , 0.7 )
( 1.23215 , 0.75 )
( 1.25107 , 0.8 )
( 1.26862 , 0.85 )
( 1.28501 , 0.9 )
( 1.30039 , 0.95 )
( 1.31489 , 1.0 )

};

\addlegendentry{TDNNS h0.25}
\addplot[color=cyan, mark=x,mark options={scale=0.5}] coordinates {
	( 0 , 0 )
	( 0.19933 , 0.05 )
	( 0.41032 , 0.1 )
	( 0.58843 , 0.15 )
	( 0.72509 , 0.2 )
	( 0.82859 , 0.25 )
	( 0.90828 , 0.3 )
	( 0.97114 , 0.35 )
	( 1.02191 , 0.4 )
	( 1.06384 , 0.45 )
	( 1.09915 , 0.5 )
	( 1.1294 , 0.55 )
	( 1.1557 , 0.6 )
	( 1.17889 , 0.65 )
	( 1.19957 , 0.7 )
	( 1.2182 , 0.75 )
	( 1.23516 , 0.8 )
	( 1.2507 , 0.85 )
	( 1.26507 , 0.9 )
	( 1.27842 , 0.95 )
	( 1.29091 , 1.0 )
	
};

\end{axis}
	\end{tikzpicture}}
	\resizebox{0.48\textwidth}{!}{
	\begin{tikzpicture}
		\begin{axis}[
			legend style={at={(0,1)}, anchor=north west},
			xlabel={displacement},
			ylabel={$P/P_{\mathrm{max}}$},
			]
			
			\addlegendentry{HHJ  h0.5}
			\addplot[color=blue, mark=square,mark options={scale=0.5}] coordinates {
				( 0 , 0 )
				( 0.16797 , 0.05 )
				( 0.33356 , 0.1 )
				( 0.46132 , 0.15 )
				( 0.55161 , 0.2 )
				( 0.61522 , 0.25 )
				( 0.66109 , 0.3 )
				( 0.69513 , 0.35 )
				( 0.72107 , 0.4 )
				( 0.74132 , 0.45 )
				( 0.75746 , 0.5 )
				( 0.77058 , 0.55 )
				( 0.78142 , 0.6 )
				( 0.79053 , 0.65 )
				( 0.79828 , 0.7 )
				( 0.80499 , 0.75 )
				( 0.81085 , 0.8 )
				( 0.81605 , 0.85 )
				( 0.82071 , 0.9 )
				( 0.82494 , 0.95 )
				( 0.82882 , 1.0 )
				
			};
			
			\addlegendentry{HHJ h0.25}
			\addplot[color=red, mark=*,mark options={scale=0.5}] coordinates {
				( 0 , 0 )
				( 0.16797 , 0.05 )
				( 0.33352 , 0.1 )
				( 0.4612 , 0.15 )
				( 0.55139 , 0.2 )
				( 0.61484 , 0.25 )
				( 0.66049 , 0.3 )
				( 0.69422 , 0.35 )
				( 0.71975 , 0.4 )
				( 0.7395 , 0.45 )
				( 0.75504 , 0.5 )
				( 0.76747 , 0.55 )
				( 0.77752 , 0.6 )
				( 0.78575 , 0.65 )
				( 0.79254 , 0.7 )
				( 0.79819 , 0.75 )
				( 0.80293 , 0.8 )
				( 0.80691 , 0.85 )
				( 0.81028 , 0.9 )
				( 0.81314 , 0.95 )
				( 0.81558 , 1.0 )
			};

\addlegendentry{TDNNS h0.5}
\addplot[color=magenta, mark=diamond,mark options={scale=0.5}] coordinates {
	( 0 , 0 )
	( 0.16804 , 0.05 )
	( 0.33402 , 0.1 )
	( 0.46222 , 0.15 )
	( 0.55296 , 0.2 )
	( 0.61707 , 0.25 )
	( 0.66352 , 0.3 )
	( 0.6982 , 0.35 )
	( 0.72485 , 0.4 )
	( 0.74589 , 0.45 )
	( 0.76288 , 0.5 )
	( 0.7769 , 0.55 )
	( 0.78871 , 0.6 )
	( 0.79882 , 0.65 )
	( 0.80763 , 0.7 )
	( 0.81542 , 0.75 )
	( 0.82241 , 0.8 )
	( 0.82875 , 0.85 )
	( 0.83457 , 0.9 )
	( 0.83996 , 0.95 )
	( 0.845 , 1.0 )
	
};

\addlegendentry{TDNNS h0.25}
\addplot[color=cyan, mark=x,mark options={scale=0.5}] coordinates {
	( 0 , 0 )
	( 0.16804 , 0.05 )
	( 0.33397 , 0.1 )
	( 0.46208 , 0.15 )
	( 0.55268 , 0.2 )
	( 0.61659 , 0.25 )
	( 0.66274 , 0.3 )
	( 0.69704 , 0.35 )
	( 0.72321 , 0.4 )
	( 0.74367 , 0.45 )
	( 0.75998 , 0.5 )
	( 0.77323 , 0.55 )
	( 0.78417 , 0.6 )
	( 0.79333 , 0.65 )
	( 0.8011 , 0.7 )
	( 0.80779 , 0.75 )
	( 0.81359 , 0.8 )
	( 0.8187 , 0.85 )
	( 0.82323 , 0.9 )
	( 0.82729 , 0.95 )
	( 0.83097 , 1.0 )
	
};

		\end{axis}
\end{tikzpicture}}
\includegraphics[width=0.49\textwidth]{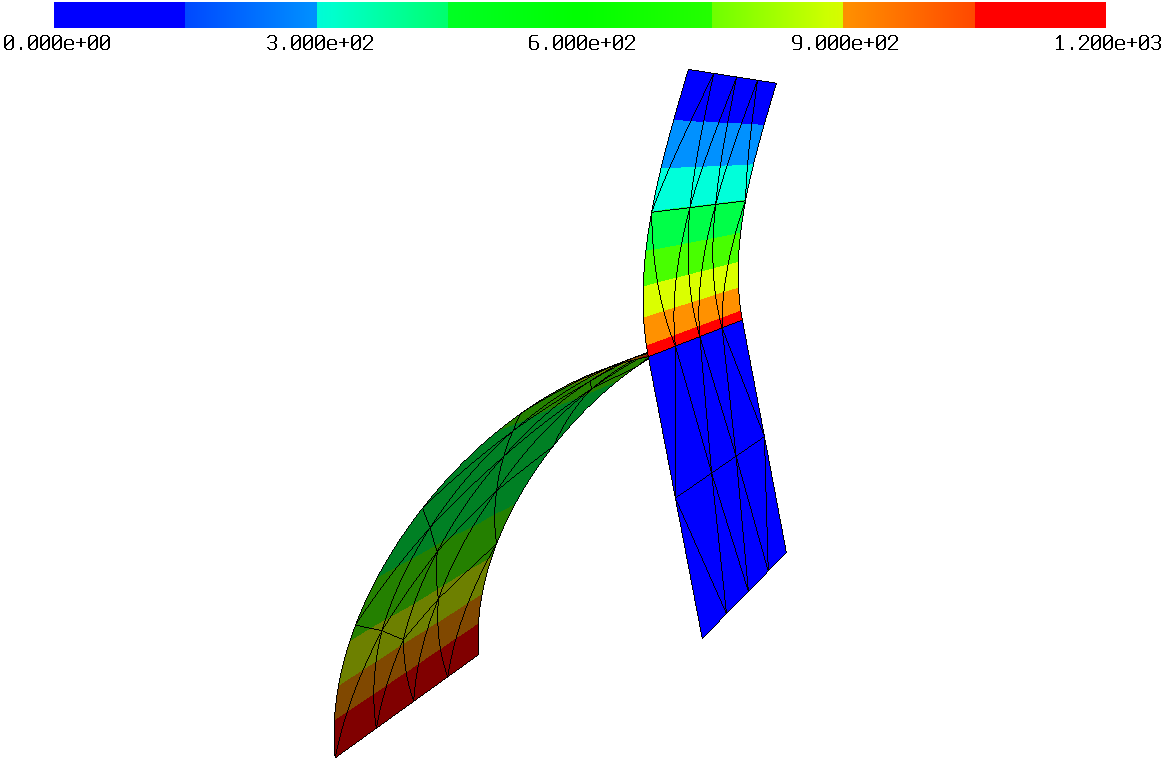}